%% file: Foliations_by_stable_spheres_with_constant_expansion__arxiv.tex
\documentclass[letterpaper]{amsart}
\input{./header__arxiv.tex}
\title[Foliations by spheres with constant expansion]{Foliations by spheres with\\constant expansion for isolated\\systems without asymptotic symmetry}
\author[Christopher Nerz]{Christopher Nerz}
\address{Mathematisches Institut\\Universit\"at T\"ubingen\\Auf der Morgenstelle 10\\72076 T\"ubingen, Germany}
\email{christopher.nerz@math.uni-tuebingen.de}
\date\today
\ifpdf
	\hypersetup{pdftitle		= {Foliation by spheres with constant expansion for isolated systems without asymptotic symmetry},
							pdfauthor		= {Christopher Roman Nerz},
							pdfsubject	= {Foliation of an initial data set by spheres with constant expansion},
							pdfkeywords	= {Constant expansion, Evolution under the Einstein equations , constant expansion foliation}}\fi
\begin{document}%
\begin{abstract}
Motivated by the foliation by stable spheres with constant mean curvature constructed by Huisken-Yau, Metzger proved that every initial data set can be foliated by spheres with constant expansion (CE) if the manifold is asymptotically equal to the standard $[t\,{=}\,0]$-timeslice of the Schwarzschild solution. In this paper, we generalize his result to asymptotically flat initial data sets and weaken additional smallness assumptions made by Metzger. Furthermore, we prove that the CE-surfaces are in a well-defined sense (asymptotically) independent of time if the linear momentum vanishes.
\end{abstract}
\maketitle
\section*{Introduction}\bgroup\def\thetheorem{\arabic{theorem}}
Motivated by an idea of Christodoulou and Yau \cite{christodoulou71some}, Huisken-Yau proved that every Riemannian manifold is (near infinity) uniquely foliated by stable surfaces with constant mean curvature (CMC) if it is asymptotically equal to the (spatial) Schwarzschild solution and has positive mass \cite{huisken_yau_foliation}. Their decay assumptions were subsequently weakened by Metzger, Huang, Eichmair-Metzger, and the author \cite{metzger2007foliations,Huang__Foliations_by_Stable_Spheres_with_Constant_Mean_Curvature,metzger_eichmair_2012_unique,nerz2014CMCfoliation}. Furthermore, the author proved that asymptotic flatness is characterized by the existence of such a CMC-foliation \cite{nerz2014GeometricCharac}.
Huisken-Yau's idea to use foliations by \lq good\rq\ hypersurfaces was picked up by Metzger who proved that every initial data set, which is asymptotically to the standard $[t\,{=}\,0]$-timeslice in the Schwarzschild solution, can (near infinity) be foliated by spheres of constant expansion (CE) and that these CE-surfaces are unique within a well-defined class of surfaces \cite{metzger2007foliations}.
He motivated the CE-foliation among other things as foliation adapted to the apparent horizons which have zero expansion and that CE-surfaces are the non-time symmetric analog of CMC-surfaces.\pagebreak[3]

Note that the CMC- and the CE-foliation are not the only foliations used in the mathematical general relativity: For example, Lamm-Metzger-Schulze achieved corresponding existence and uniqueness results for a foliation by spheres of Willmore type \cite{lamm2011foliationsbywillmore} and (in the static case) Cederbaum proved that the level-sets of the static lapse function foliate the timeslices (near infinity) \cite{Cederbaum_newtonian_limit}. However, we will only use the CMC- and the CE-foliations in this paper.\pagebreak[3]

To explain Metzger's assumptions for his existence theorem for the CE-foliation, let us first recall that an initial data set is a tuple $(\outM,\outg*,\outzFund,\outmomden*,\outenden)$ satisfying the Einstein constraint equations\footnote{We dropped the physical factor $8\pi$ for notational convenience.}
\begin{equation*}
 2\,\outenden = \outsc - \outtrtr\outzFund\outzFund + \outmc^2, \qquad
 \outmomden = \outdiv(\outmc\;\outg*-\outzFund).\labeleq{constraint_equations}
\end{equation*}
Here, $(\outM,\outg*)$ is a Riemannian manifold, $\outzFund$ is a symmetric $(0,2)$-tensor, $\outmomden$ a $(0,1)$-tensor, and $\outenden$ a function on $\outM$. This is motivated by a three-dimensional spacelike hypersurface $(\outM,\outg*)$ within a Lorenzian manifold $(\uniM,\unig*)$ with Einstein tensor $\einstein*$, the second fundamental form $\outzFund$ of $(\outM,\outg*)\hookrightarrow(\uniM,\unig*)$, its energy density $\outenden:=\einstein*(\tv,\tv)$, and the momentum density $\outmomden:=\einstein*(\tv,\cdot)$, where $\tv$ is the future pointing unit normal of $(\outM,\outg*)\hookrightarrow(\uniM,\unig*)$. If the surrounding Lorentzian manifold satisfies the Einstein equations $\einstein*=\uniric*-\frac12\,\unisc\,\unig*$, then the Gau\ss-Codazzi equations of $\outM\hookrightarrow(\uniM,\unig*)$ are equivalent to the constraint equations \eqref{constraint_equations}.

In this notation, Metzger assumed asymptotic to the standard $[t\,{=}\,0]$-timeslice of the Schwarzschild solution, i.\,e.\ the existence a coordinate system $\outx:\outM\setminus\outsymbol L\to\R^3\setminus\overline{B_1(0)}$ mapping the manifold (outside of some compact set $\outsymbol L$) to the Euclidean space (outside of a closed unit ball), such that the push forward of the metric $\outg*$ is asymptotically equal to the Schwarzschild metric $\schwarzoutg*$ as $\vert\outx\vert\to\infty$. More precisely, he assumed that the $k$-th derivatives of the difference $\outg_{ij}-\schwarzoutg_{ij}$ of the metric $\outg*$ and the Schwarzschild metric $\schwarzoutg*$ decays in these coordinates like $\vert\outx\vert^{-1-\outve-k}$ for $k\le2$.\footnote{In fact, he assumed this decay in a more geometric way: $\outg*-\schwarzoutg*=\mathcal O_0(\rad^{{-}1-\ve})$, $\outlevi_{ij}^k-\schwarzoutlevi_{ij}^k=\mathcal O_0(\rad^{-2-\ve})$, and $\outric_{ij}-\schwarzoutric_{ij}=\mathcal O_0(\rad^{-3-\ve})$ for some $\ve>0$, where we used the notation explained in Section~\ref{Assumptions_and_notation}. Actually, he also allowed $\ve=0$ if the corresponding constant is sufficiently small.} This is abbreviated with $\outg*-\schwarzoutg*=\mathcal O_2(\vert x\vert^{-1-\ve})$. He furthermore assumed that the second fundamental form $\outzFund*$ decays sufficiently fast, i.\,e.\ $\outzFund*=\mathcal O_1(\vert x\vert^{-2})$, and that the corresponding constant is sufficiently small, i.\,e.\ $\vert\hspace{.05em}\outzFund_{ij}\vert\le\nicefrac\eta{\vert x\vert^2}$ for some small constant $\eta\ll1$ and correspondingly for the first derivative. He motivated this point-wise assumption by the fact that at least in a specific example this foliation only exists for sufficiently small second fundamental form. However, this example of a second fundamental form is solely controlled by an \emph{integral} quantity: the ADM-linear momentum defined by Arnowitt-Deser-Misner \cite{arnowitt1961coordinate}. In the last paragraph of \cite{metzger2007foliations}, Metzger clarifies that (in the general setting) this foliation is nevertheless \emph{not} characterized by the ADM-linear momentum (or the ADM-mass), i.\,e.\ smallness of the linear momentum is (in general) not sufficient to ensure existence of the CE-foliation.\smallskip

The first main result of this paper is the existence of the CE-foliations under weaker decay assumptions on the metric. Furthermore, we only assume that the second fundamental form is of order $\vert\outx\vert^{-2}$, has asymptotically vanishing divergence $\outmomden*=\outdiv(\outH\outg*-\outzFund*)=\mathcal O_0(\vert\outx\vert^{-3-\ve})$, and need only additionally \lq smallness\rq\ for some integral-quantities of $\outzFund$.\footnote{Note that we can alter the assumptions on $\outzFund*$, see Remark~\ref{Alternative_assumptions}.} Here, we only state a simpler, less general version -- see Theorem~\ref{Existence} for the more general version.
\begin{corollary}[Existence of CE-foliation -- special case of Theorem~\ref{Existence}]\label{Existence_descriptive}
Let $(\outM,\outg*,\outzFund,\outmomden*,\outenden)$ be a \emph{$\Ck^2_{\frac12+\outve}$-asymptotically flat}, \emph{asymptotically maximal} initial data set with non-vanishing mass $\mass\neq0$. Assume that $\outzFund*$ is $\Ck^0_{2+\outve}$-asymptotically anti-symmetric and vanishes $\Ck^1_2$-asymptotically. If the \normalbrace*{ADM-}\linebreak[1]linear momentum is sufficiently small, then there exist closed CE-surfaces $\M<\Hradius>$ smoothly foliating $\outM$ outside a compact set.
\end{corollary}
The definition of a \lq $\Ck^2_{\frac12+\outve}$-asymptotically flat initial data set\rq\ is explained in Definition~\ref{Ck_asymptotically_flat_initial_data_set}, while the other assumptions are explained in Theorem~\ref{Existence}.\pagebreak[1] We note that the corresponding theorem is true for a temporal foliation (see Definition~\ref{Ck_asymptotically_flat_foliation}) instead of an initial data set, i.\,e.\ every asymptotically flat temporal foliation of a four-dimensional Lorentzian-manifold can be foliated (near infinity) by surfaces with constant expansion with respect to the corresponding timeslice (Theorem~\ref{Regularity_over_time}).\pagebreak[3]\smallskip

As Metzger, we also get a uniqueness result for the CE-spheres (Theorem~\ref{Uniqueness}). Again, we give a simple version -- see Theorem~\ref{Uniqueness} for the general version.
\begin{corollary}[Uniqueness of CE-surfaces -- special case of Theorem~\ref{Uniqueness}]\label{Uniqueness_descriptive}
Let $(\outM,\outg*,\outx,\outzFund,\outmomden*,\outenden)$ satisfy the assumptions of Corollary~\ref{Existence_descriptive}. Let $\M_1,\M_2\hookrightarrow\outM$ be CE-surfaces satisfying \normalbrace{specific} estimates. If $\M_1$ and $\M_2$ have the same, sufficiently small expansion, then they coincide.
\end{corollary}
The precise formulation of the \lq(specific) estimates\rq\ can be found in Theorem~\ref{Uniqueness}. Furthermore, we can again reduce the assumptions on the initial data set -- see Theorem~\ref{Uniqueness}. \pagebreak[3]\smallskip

Finally, we study how these CE-surfaces evolve in time under the Einstein equations (Theorem~\ref{Invariance}): We prove that the CE-spheres are in a well-defined sense (asymptotically) independent of time if the ADM-linear momentum vanishes. This is to be expected as the author proved that the CMC-leaves (asymptotically) evolve in time by translating in direction of the fraction of the (ADM) linear momentum and the (ADM) mass \cite[Theorem 4.1]{nerz2013timeevolutionofCMC} and the CE-spheres are asymptotically just shifts of the CMC spheres (due to the results in Section~\ref{existence}) -- and it seems appropriate to assume that this shift is (asymptotically) independent of time. To the best knowledge of the author, this is the first time that any evolution result is proven for the CE-leaves. Theorem~\ref{Invariance} implies the following (more descriptive) corollary.
\begin{corollary}[Time invariance of CE-surfaces -- special case of Theorem~\ref{Invariance}]\label{Invariance_descriptive}
Let $(\outM[t],\outg[t]*,\outzFund[t],\outenden[t],\outmomden[t])_t$ be a \normalbrace{orthogonal}** $\Ck^2_{\frac12+\outve}$-asymptotically flat temporal foliation solving the Einstein equations in asymptotic vacuum. Assume that each of these initial data sets satisfies the assumptions of Corollary~\ref{Existence_descriptive}. If the time-lapse function is $\Ck^2_{\frac12+\outve}$-asymp\-totic to $1$ and \normalbrace{asymp\-to\-ti\-cally} symmetric, then the leaves of the CE-foliation evolve \normalbrace{asymp\-to\-ti\-cally} in time by a shift in time direction \normalbrace+{orthogonal to each timeslice $\outM[t]$}.\pagebreak[2]
\end{corollary}

\textbf{Acknowledgment.}
The author wishes to express gratitude to Gerhard Huisken for suggesting this topic and many inspiring discussions. Further thanks is owed to Lan-Hsuan Huang for suggesting the use of the Bochner-Lichnerowicz formula in this setting (see Proposition~\ref{stability_operator_invertible}).
Finally, thanks goes to Carla Cederbaum for exchanging interesting thoughts about CMC- and CE-foli\-ations and about the interpretation of the integral quantities (see \eqref{assumptions_Theorem_1}, \eqref{assumptions_Theorem_2}, and Proposition~\ref{linear_Momentum_small}).\pagebreak[3]

\section*{Structure of the paper and main proof structure}
In Section~\ref{Assumptions_and_notation}, we fix the notations and basic assumptions made in this paper. Note that our assumptions on the decay of the second fundamental form is more restrictive than the one used in parts of the literature, e.\,g.\ \cite{nerz2013timeevolutionofCMC,Huang__Foliations_by_Stable_Spheres_with_Constant_Mean_Curvature}, but less restrictive than other others, e.\,g.\ \cite{christodoulou1993global,metzger2007foliations}. In Section~\ref{er_stability_operators}, we characterize the linearization of the map mapping a function to the expansion of its graph. Furthermore, we explain one of the main ideas of the following proofs. The existence and uniqueness theorems are stated in full detail and proven in Section~\ref{existence}. In the last main section (Section~\ref{invariance}), we state and prove the evolution theorem.\pagebreak[3]\smallskip

As our main proof structure for the existence and uniqueness theorems is the same as the one used by Metzger in \cite{metzger2007foliations}, we briefly explain his proof. He defines an interval $\intervalI\subseteq\interval*0*1$ to be the set of all artificial times $s\in\interval*0*1$ such that there exists a foliation by CE-surfaces for $(\outM,\outg[s]*,s\,\zFund*,s\,\outmomden*,\outenden[s]*)$ instead of $(\outM,\outg,\zFund*,\outmomden*,\outenden)$ and proves by an open-closed argument that it is equal to $\interval*0*1$. Here, $\outg[s]*$ denotes the artificial metric defined by
\[ \outg[s]_{ij} := \schwarzoutg_{ij} + s\,(\outg_{ij} - \schwarzoutg_{ij}), \qquad\schwarzoutg_{ij} := (1+\frac{\mass}{2\rad})^4\,\eukoutg_{ij} \]
and $2\,\outenden[s]:=\outsc[s] - \vert s\,\outzFund\vert^2_{\outg[s]*} + s^2\,\outmc^2$ is an artificial energy-density. As it is well-known that there is a CMC-foliation of the Schwarzschild standard-slice, he gets $0\in\intervalI$. Hence, $\intervalI$ is non-empty. He proves with a convergence argument that $\intervalI$ is also closed. In order to prove that $\intervalI$ is open, he uses the implicit function theorem: Let therefore be $\{\M[s_0]<r>\}_{r>r_0}$ be the CE-foliation of  $(\outM,\outg[s_0]*,s_0\,\zFund*,s_0\,\outmomden*,\outenden[s_0]*)$ with $s_0\in\intervalI$ and define the map
\[ (\,\!^{\boldsymbol\cdot}\textbf H \pm\,\boldsymbol\cdot\,\textbf P)^{\nu} : \interval*0*1\times\Wkp^{2,p}(\M[s_0]<r>) \to \Lp^p(\M[s_0]<r>) : (s,f) \mapsto (\H[s] \pm s\,\tr[s]\,\outzFund)(\graphnu\nu f) \]
for every $r>r_0$ and any $p>2$, where $(\H[s] \pm s\,\tr[s]\,\outzFund)(\graphnu\nu f)$ denotes the expansion of $\graphnu\nu f$ with respect to $\outg[s]*$ and $s\,\zFund*$ (see Section~\ref{er_stability_operators} for more detailed information). Assume for a moment that the Fr\'echet derivative of this map with respect to the second component at $(s_0,0)$ is invertible. The implicit function theorem then implies that $\M[s_0]<r>$ can be deformed to a surface $\M[s]<r>$ which is a CE-surface with respect to $(\outM,\outg[s]*,s\,\zFund*,s\,\outmomden*,\outenden[s]*)$ if $\vert s-s_0\vert$ is small enough. This proves the openness of $\intervalI$ under the assumption of invertibility of the Fr\'echet derivative of the above map. To deduce this invertibility, he proves multiple estimates for the distance of such a CE-surface to the origin and the trace free part $\zFundtrf[s\,]<r>*$ of the second fundamental form $\zFund[s]<r>*$ of the surface $\M[s]<r>\hookrightarrow(\outM,\outg[s]*)$. Here, he uses the concrete form of the Ricci-curvature of the Schwarzschild metric and the assumed smallness of $\outzFund*$.\smallskip

We use the same approach, but replace three main arguments:
\begin{itemize}[noitemsep,nosep]
\item as we know that the CMC-foliation of $(\outM,\outg*,\outx)$ exists \cite[Thm~3.1]{nerz2014CMCfoliation}, we can fix the metric $\outg*$ instead of using the above family of metric $\{\outg[\tau]*\}_\tau$;\footnote{Note that the proof of \cite[Thm~3.1]{nerz2014CMCfoliation} uses the explained methode including the family of metric $\{\outg[s]*\}_s$.}
\item we get the cruical estimate for the distance of $\M[\tau]<\Hradius>$ to the coordinate  origin by estimating its $\tau$-derivative (see the Lemmas~\ref{decay_rnu_full} and \ref{Derivative_estimates});
\item to conclude the invertibility of the Fr\'echet derivative explained above, we use the Bochner-Lichnerowics formula and smallness of specific integral quantities of $\outzFund*$ instead of the concrete form of the Ricci-curvature of the Schwarzschild metric and the pointwise smallness of $\outzFund$ (see Proposition~\ref{stability_operator_invertible}).\pagebreak[3]
\end{itemize}\egroup

\section{Assumptions and notation}\label{Assumptions_and_notation}
In this section, we describe the notations and decay assumptions used in this paper. The notations used are the same as used by the author in \cite{nerz2013timeevolutionofCMC,nerz2014CMCfoliation}. The assumptions on the Riemannian manifold are identical to the one e.\,g.\ described in \cite[Sec.~1]{nerz2014CMCfoliation}. The assumptions made on the other quantities of the initial data set (respectively temporal foliation) are described in Definition~\ref{Ck_asymptotically_flat_initial_data_set} (respectively Definition~\ref{Ck_asymptotically_flat_foliation}).

In order to study temporal foliations of four-dimensional spacetimes by three-dimensional spacelike slices and foliations (near infinity) of those slices by two-dimensional spheres, we will have to deal with different manifolds (of different or the same dimension) and different metrics on these manifolds, simultaneously. To distinguish between them, all four-dimensional quantities like the Lorentzian spacetime $(\uniM,\unig*)$, its Ricci and scalar curvatures $\uniric$ and $\unisc$, and all other derived quantities will carry a hat. In contrast, all three-dimensional quantities like the spacelike slices $(\outM,\outg*)$, its second fundamental form $\outzFund*$, its Ricci, scalar, and mean curvature $\outric*$, $\outsc$, and $\outH:=\outtr\,\outzFund$, its future-pointing unit normal $\outnu$, and all other derived quantities carry a bar, while all two-dimensional quantities like the CMC leaf $(\M,\g*)$, its second fundamental form $\zFund*$, the trace-free part of its second fundamental form $\zFundtrf*=\zFund*-\frac12(\tr\zFund*)\g*$, its Ricci, scalar, and mean curvature $\ric*$, $\scalar$, and $\H=\tr\zFund*$, its outer unit normal $\nu$, and all other derived quantities carry neither.

In Sections~\ref{er_stability_operators} and \ref{invariance}, the upper left index denotes the time-index $t$ of the \lq current\rq\ timeslice. In Section~\ref{existence}, it denotes the weight $\gewicht$. The only exceptions are the upper left indices $\euclideane$ and $\schws$ which refer to Euclidean and Schwarz\-schild quantities, respectively.

If different two-dimensional manifolds in one three-dimensional initial data set $(\outM,\outg*,\outzFund*,\outmomden*,\outenden)$ are involved, then the lower left index always denotes the radius $\rradius$ or curvature index $\Hradius$ of the current leaf $\M<\Hradius>$, i.\,e.\ the leaf with expansion $\H<\Hradius>\pm\tr<\Hradius>\outzFund=\nicefrac{{-}2}\Hradius$, where $\pm\in\{{-}1,+1\}$ always denotes a fixed sign. Furthermore, the two-dimensional manifolds and metrics (and other quantities) \lq inherit\rq\ the upper left index of the corresponding three-dimensional manifold. We abuse notation and suppress these indices, whenever it is clear from the context which metric we refer to.

Here, we interpret the second fundamental form and the normal vector of a hypersurface as quantities of the surface (and thus as \lq lower\rq-dimensional). For example, if $\outM[t]$ is a hypersurface in $\uniM$, then $\outnu[t]$ denotes its unit normal (and \emph{not} ${}^{t\!}\unisymbol\vartheta$). The same is true for the \lq lapse function\rq\ and the \lq shift vector\rq\ of a hypersurfaces arising as a leaf of a given deformation or foliation.\pagebreak[1]

Finally, we use upper case Latin indices $\ii$, $\ij$, $\ik$, and $\il$ for the two-dimensional range $\lbrace2,3\rbrace$ and lower case Latin indices $\oi$, $\oj$, $\ok$, and $\ol$ for the three-dimensional range $\lbrace 1,2,3\rbrace$. The Einstein summation convention is used accordingly. \pagebreak[3]\smallskip

As mentioned, we frequently use foliations and evolutions. These are infinitesimally characterized by their lapse functions and their shift vectors.
\begin{definition}[Lapse functions, shift vectors]
Let $\theta>0$, $\sigma_0\in\R$ be constants, $I\supseteq\interval{\Hradius_0-\theta\,\Hradius}{\Hradius_0+\theta\,\Hradius}$ be an interval, and $(\outM,\outg*)$ be a Riemannian manifold. A smooth map $\tensor\Phi:I\times\M\to\outM$ is called \emph{deformation} of the closed hypersurface $\M=\M<\Hradius_0>=\Phi(\Hradius_0,\M)\subseteq\outM$, if $\tensor\Phi<\Hradius>({\hspace{.05em}\cdot\hspace{.05em}}):=\Phi(\Hradius,{\hspace{.05em}\cdot\hspace{.05em}})$ is a diffeomorphism onto its image $\M<\Hradius>:=\tensor\Phi<\Hradius>(\M)$ and $\tensor\Phi<\Hradius_0>\equiv\id_{\M}$. The decomposition of $\partial*_\Hradius\Phi$ into its normal and tangential parts can be written as
\[ \partial[\Hradius]@\Phi = \rnu<\Hradius>\,\nu<\Hradius> + \rbeta<\Hradius>, \]
where $\nu<\Hradius>$ is the outer unit normal to $\M<\Hradius>$, and $\rbeta<\Hradius>\in\X(\M<\Hradius>)$ is a vector field. The function $\rnu<\Hradius>:\M<\Hradius>\to\R$ is called the \emph{lapse function} and $\rbeta<\Hradius>$ is called the \emph{shift} of $\Phi$. If $\Phi$ is a diffeomorphism (resp.\ diffeomorphism onto its image), then it is called a \emph{foliation} (resp.\ a \emph{local foliation}).\pagebreak[1]

In the setting of a Lorentzian manifold $(\uniM,\unig*)$ and a non-compact, spacelike hypersurface $\outM\subseteq\uniM$, the notions of deformation, foliation, lapse $\ralpha$, and shift $\outbeta$ are defined correspondingly.\pagebreak[3]
\end{definition}

As there are different definitions of \lq asymptotically flat\rq\ in the literature, we now give the decay assumptions used in this paper. 
To rigorously define these and to shorten the statements in the following, we distinguish between the case of a Riemannian manifold, the one of a initial data set, and the one of a temporal foliation.
\begin{definition}[\texorpdfstring{$\Ck^2_{\frac12+\outve}$}{C-2}-asymptotically flat Riemannian manifolds]\label{C2-asymptotically_flat_RM}
Let $\outve\in\interval0*{\nicefrac12}$ be a constant and let $(\outM,\outg*)$ be a smooth Riemannian manifold. The tuple $(\outM,\outg*,\outx)$ is called \emph{$\Ck^2_{\frac12+\outve}$-asymptotically flat Riemannian manifold} if $\outx:\outM\setminus\overline L\to\R^3\setminus\overline{B_1(0)}$ is a smooth chart of $\outM$ outside a compact set $\overline L\subseteq\outM$ such that
\begin{equation*}
 \vert\outg_{ij}-\eukoutg_{ij}\vert + \rad\vert\outlevi_{ij}^k\vert + \rad^2\vert\outric_{ij}\vert + \rad^{\frac52}\vert\outsc\vert \le \frac\oc{\rad^{\frac12+\outve}} \qquad\forall\,i,j,k\in\{1,2,3\} \labeleq{decay_g}
\end{equation*}
holds for some constant $\oc\ge0$, where $\eukoutg*$ denotes the Euclidean metric. Arnowitt-Deser-Misner defined the \emph{\normalbrace{ADM-}mass} of such a manifold $(\outM,\outg*,\outx)$ by
\begin{equation*} \mass_{\text{ADM}} := \lim_{\rradius\to\infty} \frac1{16\pi}\sum_{\oj=1}^3 \int_{\sphere^2_\rradius(0)}(\partial[\outx]_\oj@{\g_{\oi\oj}}-\partial[\outx]_\oi@{\g_{\oj\oj}})\,\nu<\rradius>^\oi \d\mug<\rradius>  \label{Definition_of_mass_ADM}, \end{equation*}
where $\nu<\rradius>$ and $\mug<\rradius>$ denote the outer unit normal and the area measure of $\sphere^2_\rradius(0)\hookrightarrow(\outM,\outg*)$ \cite{arnowitt1961coordinate}.
\end{definition}

In the literature, the ADM-mass is characterized using the curvature of $\outg*$:
\begin{equation} \mass := \lim_{\rradius\to\infty} \frac{{-}\rradius}{8\pi} \int_{\sphere^2_\rradius(0)} \outric*(\nu<\rradius>,\nu<\rradius>) - \frac\outsc2 \d\mug<\rradius>,\pagebreak[3] \label{Definition_of_mass}\end{equation}
see the articles of Ashtekar-Hansen, Chru\'sciel, and Schoen \cite{ashtekar1978unified,schoen1988existence,chrusciel1986remark}.
Miao-Tam recently gave a proof of this characterization, $\mass_{\text{ADM}}=\mass$, in the setting of a $\Ck^2_{\frac12+\outve}$-asymptotically flat manifold \cite{miao2013evaluation}.\footnote{The author thank Carla Cederbaum for bringing his attention to Miao-Tam's article \cite{miao2013evaluation}.} We recall that this mass is also characterized by
 \begin{equation*} \mass = \lim_{\rradius\to\infty} \mHaw(\sphere^2_\rradius(0)) \tag{\ref{Definition_of_mass}'}\label{Definition_of_mass_dash}. \end{equation*}
This can be seen by a direct calculation using the Gau\ss\ equation, the Gau\ss-Codazzi equation, and the decay assumptions on metric and curvatures. Here, $\mHaw(\sphere^2_\rradius(0))$ denotes the Hawking-mass which is for any closed hypersurface $\M\hookrightarrow(\outM,\outg*)$ defined by
\[ \mHaw(\M) := \sqrt{\frac{\volume{\M}}{16\pi}}(1-\frac1{16\pi}\int_{\M}\H^2\d\mug), \]
where $\H$ and $\mug$ denote the mean curvature and measure induced on $\M$, respectively \cite{hawking2003gravitational}.
\begin{definition}[\texorpdfstring{$\Ck^2_{\frac12+\outve}$}{C-2}-asymptotically flat initial data sets]\label{Ck_asymptotically_flat_initial_data_set}
Let $\outve\in\interval0*{\nicefrac12}$ be a constant and let $(\outM,\outg*,\outzFund,\outmomden*,\outenden)$ be an initial data set, i.\,e.\ $(\outM,\outg*)$ is a Riemannian manifold, $\outzFund$ a symmetric $(0,2)$-tensor, $\outenden$ a function, and $\outmomden$ a one-form on $\outM$, respectively, satisfying the \emph{Einstein constraint equations} \eqref{constraint_equations}.
The tuple $(\outM,\outg*,\outx,\outzFund,\outmomden*,\outenden)$ is called \emph{$\Ck^2_{\frac12+\outve}$-asymptotically flat} if $(\outM,\outg*,\outx)$ is a $\Ck^2_{\frac12+\outve}$-asymptotically flat Riemannian manifold and\vspace{-.25em}
\begin{equation*}
 \rad\,\vert\hspace{.05em}\outzFund_{ij}\vert + \rad^2\,\vert \partial[\outx]_k@{\outzFund_{ij}} \vert + \rad^{\frac52}\,\vert\outmomden_i\vert \le \frac{\oc}{\rad^{\frac12+\outve}} \qquad\forall\,i,j,k\in\{1,2,3\} \labeleq{decay_outzFund},
\end{equation*}
holds in the coordinate system $\outx$.\pagebreak[1]
In this setting, the \emph{second fundamental form $\outzFund*$ vanishes $\Ck^1_2$-asymptotically} and $\mathcal C^0_{2+\outve}$-asymptotically anti-symmetric if additionally
\[ \vert\hspace{.05em}\outzFund_{ij}\vert + \rad\,\vert \partial[\outx]_k@{\outzFund_{ij}} \vert \le \frac{\oc}{\rad^2}
		\qquad\text{and}\qquad
		\big|\hspace{.05em}\outzFund_{ij}(x) + \outzFund_{ij}({-}x)\big| \le \frac{\oc}{\rad^{2+\outve}}, \]
respectively.
\end{definition}

\begin{definition}[\texorpdfstring{$\Ck^2$}{C-2}-asymptotically flat temporal foliations]\label{Ck_asymptotically_flat_foliation}
Let ${\outve}>0$. The level sets $\outM[t]:=\time^{-1}(t)$ of a smooth function $\time$ on a four-dimensional smooth Lorentzian manifold $(\uniM,\unig*)$ are called a (orthogonal) \emph{$\Ck^2_{\frac12+\outve}$-asymptotically flat temporal fo\-li\-a\-tion}, if the gradient of $\time$ is everywhere time-like, i.\,e.\ $\unig*(\unilevi*\time,\unilevi*\time)<0$, and there is a chart $(\time,\unisymbol x):\unisymbol U\subseteq\uniM\to\R\times\R^3$ of $\uniM$, such that $(\outM[t],\outg[t]*,{\outx[t]}:=\unisymbol x|_{\outM[t]},\outzFund[t]*,\outenden[t],\outmomden[t])$ is a $\Ck^2_{\frac12+\outve}$-asymptotically flat initial data set for all $t$ (with not necessarily uniform constants $\oc[t]$), $\partial*_t|_{\outM[t]}$ is orthogonal to $\outM[t]$ for every $t$,\footnote{Here, the orthogonality of $\partial*_t|_{\outM[t]}$ to $\outM[t]$ is in fact not an additional assumption, as any coordinate system $(t,\outx)$ can be deformed (using flows in direction of $\unilevi\time$) such that this orthogonality holds for the new coordinate system.}
and the time-lapse function $\ralpha[t]:=\vert\unig*(\partial*_t,\partial*_t)\vert^{\nicefrac12}\in\Ck^1(\outM[t])$ satisfies
\begin{equation*}
 \vert\ralpha[t]-1\vert + \rad\,\vert\partial[\outx]_i@{\hspace{.05em}\ralpha[t]}\vert \le \frac{\oc[t]}{\rad^{\frac12+\outve}}, \qquad \forall\,i\in\{1,2,3\}. \labeleq{decay_lapse}
\end{equation*}
Here, the corresponding second fundamental form ${\outzFund[t]*}$, the energy-density ${\outenden[t]}$, and the momentum density ${\outmomden[t]}$ are defined by
\begin{equation*}
 \outzFund[t]_{ij} := \frac{{-}1}{2\,\ralpha[t]}\partial[t]@{\outg[t]_{ij}}, \quad\,
 \outenden[t] := \uniric*(\outnu[t],\outnu[t])+\frac\unisc2, \quad\,
 \outmomden[t]* := \uniric*(\outnu[t],\cdot), \quad\,
 \left.\partial*_\time\right|_{\outM[t]} = \ralpha[t]\;\outnu[t],
\end{equation*}
respectively, where $\outnu[t]$ is the future-pointing unit normal to $\outM[t]$. If the constants ${\oc[t]}$ of the above decay assumptions can be chosen independently of $t$, then the temporal fo\-li\-a\-tion is called \emph{uniformly $\Ck^2_{\frac12+\outve}$-asymptotically flat}.\pagebreak[1]
\end{definition}
\begin{remark}[Weaker decay assumptions]
We note that all the following results remain true in the case that the above decay assumptions are only satisfied for $\rad\,f(\rad)$ instead of $\rad^{-\ve}$, i.\,e.\ if we replace the right hand side of \eqref{decay_g}, \eqref{decay_outzFund}, and \eqref{decay_lapse} by $\rad^{\frac12}\,f(\rad)$, where $f\in\Lp^1(\interval1\infty)$ is some smooth function with $\rad\,f(\rad)\to0$ for $\rad\to\infty$.\footnote{Furthermore, we have to assume $0\ge f'(\rad)\ge\nicefrac{{-}1}\rad$, but if the above assumptions are satisfied for some $f$, then there exists a $\tilde f$ satisfying the above assumptions and this additional assumptions.} Furthermore, we can replace our pointwise assumptions by Sobolev assumptions, namely $\outg*-\eukoutg*\in\Wkp^{3,p}_{\frac12}(\R^3\setminus B_1(0))$, $\outsc\in\Lp^1(\outM)$, and $\outzFund\in\Wkp^{2,p}(\R^3\setminus B_1(0))$, where $p>2$ and where we used Bartnik's definition of weighted Sobolev spaces \cite{bartnik1986mass} -- compare with \cite[Rem.~1.2]{nerz2014CMCfoliation}.\pagebreak[3]
\end{remark}

Using one of DeLellis-M\"uller's results \cite[Thm~1.1]{DeLellisMueller_OptimalRigidityEstimates}, the author proved in \cite[Prop.~2.4]{nerz2014CMCfoliation} (see Proposition~\ref{Regularity_of_surfaces_in_asymptotically_flat_spaces}) that every closed hypersurface which is \lq almost\rq\ concentric and has \lq almost\rq\ constant mean curvature is \lq almost\rq\ umbilic, see Proposition~\ref{Regularity_of_surfaces_in_asymptotically_flat_spaces}. Here, we call a surface \lq almost concentric\rq\ if it is an element of the following class of surfaces.
\begin{definition}[\texorpdfstring{$\mathcal C_\eta(\ccenterz)$-a}Asymptotically concentric surfaces\texorpdfstring{ ($\mathcal A^{\ve,\eta}_\Aradius(\ccenterz,\c_1)$)}\relax]\label{Not_of-center}
Let $(\outM,\outg*,\outx)$ be a $\Ck^2_{\frac12+\outve}$-asymptotically flat three-dimensional Riemannian manifold and $\eta\in\interval0*1$, $\ccenterz\in\interval*01$, and $\c_1\ge0$ be constants. A closed, oriented hypersurfaces $(\M,\g*)\hookrightarrow(\outM,\outg*)$ of genus $\genus$ is called \emph{$\mathcal C_\eta(\ccenterz)$-asymptotically concentric with area radius $\Aradius=\sqrt{\nicefrac{\volume{\M}}{4\pi}}$} \normalbrace*{and constant $\c_1$}*, in symbols $\M\in\mathcal A^{\ve,\eta}_\Aradius(\ccenterz,\c_1)$ if
\begin{equation*}
 \vert\centerz\,\vert \le \ccenterz\,\Aradius+\c_1\,\Aradius^{1-\eta},\quad\
 \Aradius^{4+\eta} \le \min_{\M}\rad^{5+2\outve},\quad\
 \int_{\M}\H^2 \d\mug - 16\pi\,(1-\genus) \le \frac{\c_1}{\Aradius^\eta}, \labeleq{Not_of-center_ass}
\end{equation*}
where $\centerz=(\centerz_i)_{i=1}^3\in\R^3$ denotes the \emph{Euclidean coordinate center} defined by
\[ \centerz_i := \fint_{\M}\outx_i\d\eukmug := \frac1{\volume{\M}}\int_{\M}\outx_i\d\eukmug, \]
where $\eukmug$ denotes the measure induced on $\M$ by the Euclidean metric $\eukoutg*$ (with respect to $\outx$).\pagebreak[3]
\end{definition}
As it results in additional technical difficulties, we note that we cannot restrict ourselves to \lq really\rq\ asymptotically concentric surfaces, i.\,e.\ $\mathcal C_1(0)$-asymptotically concentric surfaces, as the CMC-surfaces used in this work are not necessarily within this class \cite{cederbaumnerz2013_examples}.

Finally, we specify the definitions of Lebesgue and Sobolev norms, we will use throughout this article.
\begin{definition}[Lebesgue and Sobolev norms]
For every compact two-dimensional Riemannian manifold $(\M,\g*)$ without boundary, the \emph{Lebesgue norms} are defined by
\[ \Vert T\Vert_{\Lp^p(\M)} := (\int_{\M} \vert T\vert_{\g*}^p \d\mug)^{\frac1p}\quad\forall\,p\in\interval*1\infty, \qquad \Vert T\Vert_{\Lp^\infty(\M)} := \mathop{\text{ess\,sup}}\limits_{\M}\,\vert T\vert_{\g*}, \]
where $T$ is any measurable function (or tensor field) on $\M$ and $\mug$ denotes the measure induced by $\g*$. Correspondingly, $\Lp^p(\M)$ is defined to be the set of all measurable functions (or tensor fields) on $\M$ for which the $\Lp^p$-norm is finite. If $\Aradius:=\sqrt{\nicefrac{\volume{\M}}{4\pi}}$ denotes the \emph{area radius} of $\M$, then the \emph{Sobolev norms} are defined by
\[ \Vert T\Vert_{\Wkp^{k+1,p}(\M)} := \Vert T\Vert_{\Lp^p(\M)} + \Aradius\,\Vert\levi*T\Vert_{\Wkp^{k,p}(\M)}, \qquad \Vert T\Vert_{\Wkp^{0,p}(\M)} := \Vert T\Vert_{\Lp^p(\M)}, \]
where $k\in\N_{\ge0}$, $p\in\interval*1*\infty$, and $T$ is any measurable function (or tensor field) on $\M$ for which the $k$-th (weak) derivative exists. Correspondingly, $\Wkp^{k,p}(\M)$ is the set of all functions (or tensors fields) for which the $\Wkp^{k,p}(\M)$-norm is finite. Furthermore, $\Hk^k(\M)$ denotes $\Wkp^{k,2}(\M)$ for any $k\ge1$ and $\Hk(\M):=\Hk^1(\M)$.\pagebreak[3]
\end{definition}

\section{Pseudo stability operators of the expansion}\label{er_stability_operators}
In this section, we assume that $(\uniM,\unig*)$ is a Lorentzian manifold and that $\time$ is a smooth regular function on $\uniM$ such that its level sets $\outM[t]:=\time^{-1}(t)$ form a $\Ck^2_{\frac12+\outve}$-as\-ymp\-to\-tic\-al\-ly flat temporal foliation (for some $\ve>0$) -- with respect to a fixed chart $\unisymbol x$.\footnote{In fact, we do not need asymptotically flatness in this section, but only that $\outM[t]$ are spacelike hypersurfaces foliating $\uniM$ smoothly.} In particular, every level set $(\outM[t],\outg[t]*,\outx[t],\outzFund[t],\outenden[t],\outmomden[t])$ is a three-dimensional initial data set, where we used the same notation as in Definition~\ref{Ck_asymptotically_flat_foliation}. Furthermore, let $\M\hookrightarrow\outM$ be a smooth, closed hypersurface in one of the timeslices $\outM:=\outM[t_0]$.\smallskip

The (conventional) stability operator $\jacobiext*$ of $\M\hookrightarrow\outM$ is well-understood. It can be defined as the linearization of the mean curvature map
\begin{equation*} \textbf H : \Ck^2(\M) \to \Ck^0(\M) : f \mapsto \H(\graphnu\nu f) \labeleq{mean_curvature_map} \end{equation*}
at $f\equiv0$, where $\H(\graphnu\nu f)$ is the mean curvature of
\begin{equation*} \graphnu\nu f := \lbrace \outexp_p(f(p)\,\nu) \ \middle|\ p\in\M\rbrace \labeleq{graph_nu} \end{equation*}
with respect to the surrounding metric $\outg$. Here, $\outexp_p$ denotes the exponential map of $\outM$ at a point $p\in\M$. This graph is a well-defined closed hypersurface if $\M$ is smooth and $f$ lays in some well-defined $\Lp^2(\M)$-neighborhood of zero (depending on $\M$ and $\outM$). In particular, the (conventional) stability operator is well-defined for every smooth, closed hypersurface $\M\hookrightarrow\outM$ and it is well-known that it is characterized by
\begin{equation*} \jacobiext*f = \laplace f + (\outric*(\nu,\nu) + \trtr\zFund\zFund)\,f \qquad\forall\,f\in\Ck^2(\M). \labeleq{stability_operator} \end{equation*}

As we want to construct surfaces with constant expansion $\H\pm\troutzFund$ (and not with constant mean curvature), it is intuitive to replace the mean curvature map by the \lq expansion map\rq
\begin{equation*} \textbf H \pm\textbf P : \Ck^2(\M) \to \Ck^0(\M) : f \mapsto (\H \pm \troutzFund)(\graphnu? f) \labeleq{expansion_rate_map} \end{equation*}
as it was already done by Metzger \cite{metzger2007foliations}\footnote{Note that Metzger considered $\graphnu\nu f$, i.\,e.\ the graph in spatial direction.}, where $\pm$ denotes a fixed sign and $(\H \pm \troutzFund)(\graphnu? f)$ denotes the expansion of the graph of $f$ in \lq some direction\rq\ -- in the mean curvature case \eqref{mean_curvature_map}, this direction was the \emph{spatial direction} $\nu$. We recall that the expansion is the mean curvature of this graph within its future (or past) expanding light-cone and that it is given by the mean curvature $\H$ of this graph within (the corresponding) timeslice $\outM$ plus (or minus) the two-dimensional trace of the second fundamental form $\outzFund$ of (the corresponding) timeslice $\outM$ in $\uniM$. In this case, there are multiple \lq intuitive directions\rq\ in which the graph can be constructed. In this section, we characterize two of these (by linear combination this is sufficient for any direction):

First, we linearize this map \lq within $\outM$\rq, i.\,e.\ linearize the spatial expansion map
\begin{equation*} (\textbf H \pm\textbf P)^{\nu} : \Ck^2(\M) \to \Ck^0(\M) : f \mapsto (\H \pm \troutzFund)(\graphnu\nu f). \labeleq{radial_expansion_rate_map} \end{equation*}
where $\graphnu\nu f$ is as defined in \eqref{graph_nu}, i.\,e.\ the same graph as used in the above case of the (conventional) stability operator. To the best knowledge of the author, this was first done by Metzger \cite{metzger2007foliations}. We denote this \emph{pseudo stability operator} by $\jacobipm{\pm}$. It should be noted that $\jacobipm{\pm}$ does not arise as a second variation of the area operator $f\mapsto\volume{\graphnu\nu f}$ as the (conventional) stability operator does. The reason for this is that the first the variation is done in null-direction (resulting in $\textbf H \pm\textbf P$) and the second one in spatial-direction. This operator has been studied in more detail by Andersson-Mars-Simon, Andersson-Metzger, Andersson-Eichmair-Metzger, and others, see \cite{andersson2005local,andersson2009area,andersson2011jang} and the citations therein.

Second, we linearize the corresponding map in time direction, i.\,e.\ linearize the temporal expansion map
\begin{equation*} (\textbf H \pm\textbf P)^{\outnu} : \Ck^2(\outM) \to \Ck^0(\M) : \bar f \mapsto (\H \pm \troutzFund)(\graphnu\tv\bar f), \labeleq{temporal_expansion_rate_map} \end{equation*}
where $\graphnu\tv f := \lbrace \uniexp_p(f(p)\,\tv) \,:\,  p\in\M\rbrace\hookrightarrow\outM[\bar f]:=\lbrace\uniexp_{\bar p}(\bar f(\bar p)\,\tv)\,:\,\bar p\in\outM\}$ is the graph in \lq future direction\rq. Here, $\uniexp_p$ and $\uniexp_{\bar p}$ denote the exponential map of $\uniM$ at a point $p\in\M$ and at a point $\bar p\in\outM$, respectively. This means in particular that the mean curvature and second fundamental form $\outzFund$ of $(\textbf H\pm\textbf P)(f)$ at a point $p\in\M$ are calculated with respect to the metric and second fundamental form of $\outM[\bar f]$.\footnote{Note that $(\textbf H\pm\textbf P)^{\outnu}(\bar f)$ depends not only on the values of $\bar f$ on $\M$, i.\,e.\ on $f:=\bar f|_{\M}$, but also on its $\outM$-gradient $\outlevi*\bar f|_{\M}$ and its $\outM$-hessian $\outHess*\bar f|_{\M}$.}

The linearization of $(\textbf H\pm\textbf P)^{\outnu}$ is denoted by $\jacobit{\pm}$. For notation convenience, we only calculate $\jacobit{\pm}\ralpha$, where $\ralpha$ denotes the temporal lapse function of the temporal foliation $\{\outM[t]\}_t$, i.\,e.\ we only look at functions $\bar f$ with $\graphnu\tv\bar f=\outM[t]$ for some $t$.\pagebreak[2]

As first step, we calculate some identities for the Ricci curvature of $\uniM$ on a two-dim\-en\-sio\-nal deformation $\Phi:\interval{t_0-\eta}{t_0+\eta}\times\interval-\eta\eta\times\M\to\outM$ of a closed, two-dim\-en\-sio\-nal surface $\M=\Phi(t_0,0,\M)\hookrightarrow\outM[t_0]$. Again, we restrict ourselves to the case of deformations \emph{compatible with the temporal foliation $\{\outM[t]\}_t$}, i.\,e.\ $\Phi(t,\Hradius,p)\in\outM[t]$ for every $t\in\interval{t_0-\eta}{t_0+\eta}$, $\Hradius\in\interval-\eta\eta$, and $p\in\M$. Furthermore, we assume that it is orthogonal, i.\,e.\ $\partial*_t\Phi=:\ralpha[t]\,\tv[t]$ and $\partial*_\Hradius\Phi=:\rnu[t]<\Hradius>\,\nu[t]<\Hradius>$ for some smooth function $\rnu<\Hradius>[t]$ on $\M<\Hradius>[t]:=\Phi(t,\Hradius,\M)$ and the temporal lapse function $\ralpha[t]$ on $\outM[t]$, where $\tv[t]$ again denotes the future pointing unit normal of $\outM[t]\hookrightarrow\uniM$ and $\nu[t]<\Hradius>$ is the outer unit normal of $\M[t]<\Hradius>\hookrightarrow\outM[t]$. Finally, we denote the time derivative $\partial*_t(\pullback\Phi T)$ and the spatial derivative $\partial*_\Hradius(\pullback\Phi T)$ of any quantity $T$ by $\partialt T$ and $\partialr T$, respectively.\nopagebreak[3]
\begin{proposition}[Curvature identities]
Let $(\uniM,\unig)$ be a smooth Lorentzian manifold, $\lbrace\outM[t]\rbrace_t$ be a smooth temporal foliation with respect to a smooth time function $\time$ on $\uniM$. Additionally, let 
\[ \Phi:\interval-{\theta+t_0}{\theta+t_0}\times\interval-\theta\theta\times\M\to\uniM:(t,\sigma,p)\mapsto\Phi(t,\sigma,p) \]
be a smooth, orthogonal foliation-compatible deformation of a closed hypersurface $\M\hookrightarrow\outM[t_0]=:\outM$ \normalbrace{see above}. Suppressing the indices $t_0$ and $\sigma=0$, the tensor identities
{\normalfont\begin{align*}
 \uniric*(\partial*_t,\nu)
	={}& \partialt\H + \ralpha\,(\div\outzFundnu - \trtr\zFund\outzFund) - \partialr\ralpha\,\troutzFund
			+ 2 \outzFundnu(\levi*\ralpha) \labeleq{partialtH} \\\displaybreak[1]
 \uniric*(\partial*_t,\tv)
	={}& \partialt(\troutzFund+\outzFund_{\nu\nu})-\ralpha\,(\trtr\outzFund\outzFund+4\trtr\outzFundnu\outzFundnu
			+ \outzFund_{\nu\nu}^2) + \laplace\ralpha + (\outHess\,\ralpha)(\nu,\nu)  \labeleq{uniric00}\\\displaybreak[1]
 \uniric*(\tv,\partial*_\sigma)
	={}& \partialr(\troutzFund) - 2\outzFund_\nu(\levi\rnu) + \rnu(\H\,\outzFund_{\nu\nu} - \div\outzFundnu
			- \trtr\zFund\outzFund) \labeleq{uniric01} \\\displaybreak[1]\labeleq{uniric11}
 \uniric*(\nu,\nu)
	={}& \outric*(\nu,\nu) - 2\trtr\outzFundnu\outzFundnu + \outzFund_{\nu\nu}^2 + \troutzFund\,\outzFund_{\nu\nu}
			- \frac{\partialt(\outzFund_{\nu\nu})}\ralpha - (\outHess\,\ralpha)(\nu,\nu) 
\end{align*}}%
hold on $\M$, where {\normalfont$(\trzd\zFund\outzFund)_{\ii\!\ij}:=\zFund_{\ii\!\ik}\,\g^{\ik\!\il}\,\outzFund_{\ij\!\il}$}, {\normalfont$\outzFundnu:=\outzFund(\nu,\cdot)$}, and {\normalfont$\outzFund_{\nu\nu}:=\outzFund(\nu,\nu)$}.\pagebreak[2]
\end{proposition}

\begin{proof}
The first identity \eqref{partialtH} was proven by the author in \cite[Prop.~3.7]{nerz2013timeevolutionofCMC}. By the well-known identity for the (conventional) stability operator in the Lorentzian case, we know
\begin{align*}
 \uniric*(\tv,\partial*_t) ={}& \partialt\outH + \outlaplace\ralpha - \ralpha\outtrtr\outzFund\outzFund \\
	={}& \partialt(\troutzFund+\outzFund_{\nu\nu}) + \laplace\ralpha + (\outHess\,\ralpha)(\nu,\nu)
			- \ralpha(\trtr\outzFund\outzFund + 2 \g^{\ii\!\ij}\,\outzFund_{\nu\ii}\,\outzFund_{\nu\ij} + \outzFund_{\nu\nu}^2).
\end{align*}
Thus, the identity \eqref{uniric00} is proven. Furthermore, the Codazzi equations
\[ \uniric*(\tv,\partial*_\Hradius) = (D\outH - \outdiv\outzFund)(\partial*_\Hradius) = (D\outH - \outdiv\outzFund)(\rnu\,\nu) \] 
implies
\begin{align*}
 \uniric*(\tv,\partial*_\Hradius)
	={}& \partialr(\troutzFund + \outzFund_{\nu\nu}) - \rnu^{-2}(\partialr(\rnu^2\outzFund_{\nu\nu}) - 2 \outzFund(\partialr\rnu\,\nu - \rnu\g^{\ii\!\ij}\,D_\ii\rnu\,e_\ij,\partial*_\Hradius)) \\
		& - \g^{\ii\!\ij}(\partial_\ii@{(\rnu\,\outzFund_{\nu\ij})} - \outzFund(D_p\rnu\,\nu - \rnu\,\g^{\ik\!\il}\,\zFund_{\ii\!\ik}\,e_\il,e_\ij)) \\
		& + \g^{\ii\!\ij}\,\outzFund(\partial*_\Hradius,\zFund_{\ii\!\ij}\hspace{.05em}\nu + \levi*_{\!e_\ii}e_\ij) \\
	={}& \partialr(\troutzFund) - 2 \g^{\ii\!\ij}D_\ii\rnu\;\outzFund_{\nu\hspace{-.05em}\ij} - \rnu\,\div\outzFundnu - \g^{\ii\!\ij}(\rnu\g^{\ik\!\il}\,\zFund_{\ii\!\ik}\,\outzFund_{\ij\!\il} - \rnu\zFund_{\ii\!\ij}\,\outzFund_{\nu\nu}).
\end{align*}
This proves \eqref{uniric01}. It is well-known that for any hypersurface $\outM\hookrightarrow\uniM$ the identity
\[ \ralpha\,\uniric*(\nu,\nu) = \ralpha\,(\outric*(\nu,\nu) -\,2\outtrtr\outzFundnu\outzFundnu + \outH\,\outzFund_{\nu\nu}) - (\outHess\,\ralpha)(\nu,\nu) - (\partialt\outzFund)_{\nu\nu} \]
holds and this is equivalent to \eqref{uniric11}.\pagebreak[3]
\end{proof}

As a direct corollary, we get the desired characterizations of the pseudo stability operators (see below). As explained, we are only interested how the expansion change within the given temporal foliation -- in particular, we only calculate the linearization of the temporal expansion map for the lapse function $\ralpha$ of the temporal foliation. Furthermore, we replace the sign $\pm$ in the expansion map \eqref{expansion_rate_map} for technical reasons by a factor $\gewicht\in\interval*-1*1$ -- this means, we are concerning the expansion not only in null-direction, but also in the spacelike direction $\nu+\gewicht\,\tv$ ($\gewicht\in\interval-11$). The reason for this is an open-closed argument in Section~\ref{existence}.\pagebreak[1]
\begin{corollary}[Pseudo stability operators]\label{Stability_operators}
Let $\M\hookrightarrow\outM[t_0]$ be a closed hypersurface and $\gewicht\in\interval*-1*1$ be a constant. The spatial \normalbrace{weighted} expansion pseudo stability operator $\jacobipm[t_0]$ is defined as linearization of the spatial \normalbrace{weighted}** expansion map
\[ \normalfont(\textbf H + \gewicht\,\textbf P)^{\nu} : \Ck^2(\M) \to \Ck^0(\M) : f \mapsto (\H + \gewicht\,\troutzFund)(\graphnu\nu f), \]
in $f\equiv0$ and the \normalbrace{signed}** temporal pseudo stability operator $\jacobit[t_0]$ defined as linearization of the temporal expansion map \eqref{temporal_expansion_rate_map} in $f\equiv0$, respectively. Suppressing the index $t_0$, these are characterized by
{\normalfont\begin{align*} \labeleq{jacobipm}
 \jacobipm*f
	={}& \jacobiext*f + 2\gewicht\,\outzFundnu(\levi*f) + \gewicht\,(\div\outzFundnu +{}\trtr\zFund\outzFund + \outmomden(\nu) - \H\,\outzFund_{\nu\nu})f,\\\labeleq{jacobit}
 \jacobit*\ralpha
	={}& \mp \laplace\ralpha + \rnu\,D_{\nu}\ralpha\;\troutzFund - 2 \outzFundnu(\levi*\ralpha) + (\outmomden(\nu)-\div\outzFundnu + \trtr\zFund\outzFund)\,\ralpha \\
			&	\pm (\outenden + \einstein*(\nu,\nu) + \trtr\outzFund\outzFund +\, 6\trtr\outzFundnu\outzFundnu -\,\troutzFund\;\outzFund_{\nu\nu} - \outric*(\nu,\nu))\,\ralpha
\end{align*}}%
for the temporal lapse function $\ralpha:=\vert\unig*(\unilevi*\time,\unilevi*\time)\vert^{\nicefrac{{-}1}2}$ of $\lbrace\outM[t]\rbrace_t$ at $t=t_0$ and any smooth function $f\in\Ck^2(\M)$, where $\einstein*:=\uniric*-\frac12\;\unisc\;\unig*$, $\outmomden:=\einstein*(\tv,\cdot)$, and $\outenden:=\einstein*(\tv,\tv)$ denote the Einstein tensor of $\uniM$, the momentum density of $\outM$, and the energy-density of $\outM$, respectively.
\end{corollary}

In particular, we note that the spatial (weighted) stability operators depends only on quantities of the initial data set $\outM[t_0]$ (as to be suspected) while the temporal (signed) stability operator depends on the quantity $\einstein*(\nu,\nu)$ and the lapse function $\ralpha$, i.\,e.\ on quantities of the Lorentzian manifold $\uniM$ and the temporal foliation (as to be suspected), respectively.\pagebreak[3]

\section{Existence of the CE-foliation and uniqueness of CE-spheres}\label{existence}
In this section, we prove existince of a (unique) smooth sphere $\M[\pm]<\Hradius>$ with constant expansion (CE) $\H\pm\troutzFund\equiv\nicefrac{{-}2}\Hradius$ for every three-dimensional $\Ck^2_{\frac12+\outve}$-as\-ymp\-to\-ti\-cal\-ly flat initial data set $(\outM,\outg*,\outx,\outzFund,\outmomden*,\outenden,\ralpha)$ with sufficiently fast vanishing second fundamental form if some additional integral assumptions on $\outzFund$ are satisfied. Here, $\pm$ denotes a (fixed) sign and we assume that $\Hradius$ is large enough (depending on the decay constants of the initial data set). Furthermore, we prove that these CE-spheres foliate $\outM$ outside some compact set $\outsymbol K$. More precisely, we prove the following existence and uniqueness theorems.
\begin{theorem}[Existence of the CE-foliation]\label{Existence}
Let $\outve>0$, $R_0>0$, $\outck\ge0$ be constants and $(\outM,\outg*,\outx,\outzFund,\outmomden*,\outenden)$ be a $\Ck^2_{\frac12+\outve}$-asymptotically flat initial data set with $\Ck^1_2$-asymptotically vanishing second fundamental form $\outzFund$ and non-vanishing mass $\mass\neq0$. There exists a positive constant $\outck=\Cof{\outck}[\mass][\outve][\outc]>0$ with the following property: if
\begin{align*}
 \vert\int_{\sphere^2_\rradius(0)}\outzFund_{kl}\,\frac{\outx^k}\rradius(\frac{\outx^j\,\eukoutg^{il}-\outx^i\,\eukoutg^{jl}}{\rradius}) \d\mug\vert \le{}& \outck, &
 \vert\int_{\sphere^2_\rradius(0)}\troutzFund \d\mug\vert \le{}& \outck, \labeleq{assumptions_Theorem_1} \\
 \vert\int_{\sphere^2_\rradius(0)}\outH\,\frac{\outx^i}{\rradius}\,\frac{\outx^j}{\rradius} \d\mug\vert \le \outck, \quad
 \vert\int_{\sphere^2_\rradius(0)}\outH\,\frac{\outx^i}{\rradius}\d\mug\vert \le{}& \outck, &
 \vert\int_{\sphere^2_\rradius(0)}\troutzFund\,\frac{\outx^i}{\rradius}\d\mug\vert \le{}& \outck,
\labeleq{assumptions_Theorem_2}\end{align*}
hold for every $i,j,k\in\{1,2,3\}$ and $\rradius>R_0$ for some $R_0>0$, then there exist a constant $\Hradius_0={\Cof{\Hradius_0}[\mass][\outve][\outc][R_0]}$ and two $\Ck^1$-maps $\outPhi[\pm]:\interval{\Hradius_0}\infty\times\sphere^2\to\outM$ such that $\M<\Hradius>[\pm]:=\outPhi[\pm](\Hradius,\sphere^2)$ has constant expansion $\H<\Hradius>[\pm]\pm\tr<\Hradius>[\pm]\outzFund*\equiv\nicefrac{-2}\Hradius$ for any $\Hradius>\Hradius_0$. Furthermore, these CE-surfaces foliate $\outM$ near infinity, i.\,e.\ the maps $\tensor\Phi[\pm]$ are diffeomorphisms onto their images and $\outM\setminus\outPhi[\pm](\interval{\Hradius_0}\infty\times\sphere^2)$ are both compact.
\end{theorem}
We see that the integrals in \eqref{assumptions_Theorem_1} and the first one in \eqref{assumptions_Theorem_2} vanish asymptotically if $\outzFund$ is asymptotically anti-symmetric, i.\,e.\ $\vert\hspace{.05em}\outzFund(\outx)+\outzFund({-}\outx)\vert\le\nicefrac\oc{\rad^{2+\outve}}$ implies that the integral inequalities in \eqref{assumptions_Theorem_1} and the first one in \eqref{assumptions_Theorem_2} are satisfied for every $\outck>0$ (if $R_0=\Cof{R_0}[\outck][\ve]$ is sufficiently large). In particular, these integrals vanish asymptotically if the \emph{Regge-Teitelboim conditions} are satisfied, for more information about these conditions see for example \cite{regge1974role,huang2009center}. Equally, the second inequality in \eqref{assumptions_Theorem_2} vanishes asymptotically if the initial data set is asymptotically maximal, i.\,e.\ $\vert\outH\vert\le\nicefrac\oc{\rad^{2+\outve}}$. We prove in Proposition~\ref{linear_Momentum_small} that the last integral in \eqref{assumptions_Theorem_2} asymptotically corresponds to the linear momentum.

\begin{remark}[Alternative assumptions]\label{Alternative_assumptions}
We can alter the assumptions in Theorem~\ref{Existence} on the second fundamental form $\outzFund*$: if \eqref{assumptions_Theorem_1} and \eqref{assumptions_Theorem_2} are satisfied for $\nicefrac\outck{\Hradius^\delta}$ instead of $\outck$, where $\delta\in\interval0*\ve$, then we can replace the assumption \lq $\Ck^1_2$-asymp\-to\-ti\-cally vanishing second fundamental form $\outzFund*$\rq, i.\,e.\ $\vert\hspace{.05em}\outzFund\vert_{\outg*}+\rad\,\vert\outlevi*\outzFund\vert_{\outg*}\le\nicefrac\oc{\rad^2}$, by \lq$\vert\hspace{.05em}\outzFund\vert_{\outg*}+\rad\,\vert\outlevi*\outzFund\vert_{\outg*}\le\nicefrac\oc{\rad^{2-\delta}}$\rq. However, as this does not need any additional argument, we use the assumptions explained in Theorem~\ref{Existence}.\pagebreak[3]\smallskip
\end{remark}

The corresponding result is also true for a temporal foliation instead of a single timeslice:\nopagebreak
\begin{theorem}[Regularity of the CE-surfaces in time]\label{Regularity_over_time}
Let $(\outM[t],\outg[t]*,\outx[t],\outzFund[t],\outenden[t],\outmomden[t],\ralpha[t])_{t\in I}$ be a uniformly $\Ck^2_{\frac12+\outve}$-asymptotically flat temporal foliation \normalbrace{for some $\outve>0$} such that $(\outM[t],\outg[t]*,\outtensor[t] x,\outzFund[t],\outenden[t],\outmomden[t],\ralpha[t])$ satisfies for any time $t$ the assumptions of Theorem~\ref{Existence} including \eqref{assumptions_Theorem_1} and \eqref{assumptions_Theorem_2}. There are two $\Ck^1$-maps $\uniPhi[\pm]:I\times\interval{\Hradius_0}\infty\times\sphere^2\to\uniM$ such that $\outPhi[t,\pm]:=\uniPhi[\pm](t,\cdot,\cdot)$ are the maps $\outPhi[\pm]$ from Theorem~\ref{Existence} for $(\outM[t],\outg[t]*,\outtensor[t] x,\outzFund[t],\outenden[t],\outmomden[t])$ and every $t\in I$.\pagebreak[3]
\end{theorem}
We also get the corresponding uniqueness result.\nopagebreak
\begin{theorem}[Uniqueness of the CE-surfaces]\label{Uniqueness}
Let $(\outM,\outg*,\outx,\outzFund,\outmomden*,\outenden)$ satisfy the assumptions of Theorem~\ref{Existence} including \eqref{assumptions_Theorem_1} and \eqref{assumptions_Theorem_2}, $\eta\in\interval0*1$ and $\c_1\ge0$ be constants, and $\pm$ be a fixed sign. There are constants ${\ccenterz}=\Cof{{\ccenterz}}[\outve][\oc][\eta]\in\interval01$, $\Hradius_0=\Cof{\Hradius_0}[\outve][\oc][\eta][\c_1]$ such that every closed hypersurface $\M\in\mathcal A^{\ve,\eta}_\Aradius(\ccenterz,\c_1)$ with constant expansion $\H\pm\troutzFund\equiv\nicefrac{{-}2}\Hradius$ and $\Hradius>\Hradius_0$ is the leaf $\M[\pm]<\Hradius>$ of the CE-foliation constructed in Theorem~\ref{Existence}.
\end{theorem}
We see that these existence and uniqueness theorems imply the descriptive versions (Corollaries~\ref{Existence_descriptive} and \ref{Uniqueness_descriptive}), if the second inequality in \eqref{assumptions_Theorem_2} holds under the assumptions made in these corollaries. We prove this in Proposition~\ref{linear_Momentum_small}.\smallskip\pagebreak[3]

As explained in the introduction, Huisken-Yau proved that any asymptotically Schwarzschildean three-dimensional manifold can be foliated (near infinity) by hypersurfaces with constant mean curvature (CMC) and that these CMC-surfaces satisfy strong decay assumptions \cite{huisken_yau_foliation}. Later, this was generalized by Metzger, Huang, Eichmair-Metzger, and other assuming asymptotically flatness and different asymptotically symmetry conditions on the components $\outg_{ij}$ of the metric $\outg*$ \cite{metzger2007foliations,Huang__Foliations_by_Stable_Spheres_with_Constant_Mean_Curvature,metzger_eichmair_2012_unique}.\footnote{In fact, Metzger and Eichmair-Metzger assumed that the manifold is asymptotically equal to the (spatial) Schwarz\-schild solution and Huang assumed the Regge-Teitelboim conditions, i.\,e.\ asymptotic symmetry with respect to the coordinate origin.} The author proved that these results remain true for $\Ck^2_{\frac12+\outve}$-asymptotically flat manifolds \cite[Thm~3.1]{nerz2014CMCfoliation}.
\begin{theorem}[Existence of the CMC-surfaces, {\cite[Thms~3.1, 3.2]{nerz2014CMCfoliation}}]\label{existence_of_CMC_leaf}
Let $(\outM,\outg*,\outx)$ be a $\Ck^2_{\frac12+\outve}$-asymptotically flat, three-dimensional Riemannian manifold and non-vanishing mass $\mass\neq0$ \normalbrace{for some ${\outve}>0$}. There exist constants $\Hradius_0=\Cof{\Hradius_0}[\mass][{\outve}][\oc]$ and $\c=\Cof{\c}[\mass][{\outve}][\oc]$, a compact set $\outsymbol K\subseteq\outM$, and a $\Ck^1$-diffeomorphism $\Phi:\interval{\Hradius_0}\infty\times\sphere^2\to\outM\setminus\outsymbol K$ such that each  $\M<\Hradius>:=\Phi(\Hradius,\sphere^2)$ has constant mean curvature $\H<\Hradius>\equiv\nicefrac{{-}2}\Hradius$ and satisfies $\M<\Hradius>\in\mathcal A^{\ve,\ve}_{\Aradius(\Hradius)}(0,\c)$ for every $\Hradius>\Hradius_0$, where $\Aradius(\Hradius):=\sqrt{\nicefrac{\volume{\M<\Hradius>}}{4\pi}}$.\pagebreak[3]
\end{theorem}
Furthermore, there is a corresponding uniqueness theorem for the CMC-surfaces.\nopagebreak
\begin{theorem}[Uniqueness of the CMC-surfaces, {\cite[Thm~3.3]{nerz2014CMCfoliation}}]\label{CMC_Uniqueness}
Let $(\outM,\outg*,\outx)$ be a $\Ck^2_{\frac12+\outve}$-asymptotically flat, three-dimensional Riemannian manifold with non-vanishing mass $\mass\neq0$ \normalbrace{for some ${\outve}>0$}. For every constants $\eta\in\interval0*1$, $\ccenterz\in\interval*01$, and $\c_1>0$, there is a constant $\Aradius_1=\Cof{\Aradius_1}[\mass][\outve][\oc][\eta][\ccenterz][\c_1]$ such that every closed hypersurface $\M\in\mathcal A^{\ve,\eta}_\Aradius(\ccenterz,\c_1)$ with radius $\Aradius=\sqrt{\nicefrac{\volume\M}{4\pi}}>\Aradius_1$ and constant mean curvature $\H\equiv:\nicefrac{{-}2}\Hradius$ coincides with the CMC surface $\M<\Hradius>$ constructed in Theorem~\ref{existence_of_CMC_leaf}.\pagebreak[3]
\end{theorem}
Again, we note that the corresponding results were also proven by Huisken-Yau, Metzger, Huang, Eichmair-Metzger, and others for the corresponding decay assumptions on $\outg*$ \cite{huisken_yau_foliation,metzger2007foliations,Huang__Foliations_by_Stable_Spheres_with_Constant_Mean_Curvature,metzger_eichmair_2012_unique}.\pagebreak[3]

We will use the following regularity result proven by the author in \cite[Prop.~2.4]{nerz2014CMCfoliation} -- we again note that a similiar result was proven by Metzger in the setting that the surrounding manifold $(\outM,\outg*)$ is asymptotically equal to the (spatial) Schwarzschild solution.
\begin{proposition}[Regularity of surfaces in asymp.~flat spaces, {\cite[Prop.~2.4]{nerz2014CMCfoliation}}]\label{Regularity_of_surfaces_in_asymptotically_flat_spaces}
Let $(\M,\g*)$ be a closed, oriented hypersurface in a $\Ck^2_{\frac12+\outve}$-asymptotically flat three-dim\-en\-sio\-nal Riemannian manifold $(\outM,\outg*)$ and let $\eta\in\interval0*{\outve}$, $\ccenterz\in\interval*01$, $\c_1\ge0$, and $p\in\interval2\infty$ be constants. If $\M\in\mathcal A^{\ve,\eta}_\Aradius(\ccenterz,\c_1)$ is a closed hypersurface with
\[
 \exists\,\mean\H(\M)\in\R:\qquad\ 
 \Vert \H - \mean\H(\M) \Vert_{\Wkp^{1,p}(\M)} \le \frac{\c_1}{\Aradius^{\frac32+\ve-\frac2p}},
\]
then there are constants $\Aradius_1=\Cof{\Aradius_1}[\outve][\oc][\ccenterz][\c_1][\eta][p]$ and $C=\Cof[\outve][\oc][\ccenterz][\c_1][\eta][p]$ such that $\M$ is a sphere and
\begin{equation*} \Aradius^{-1}\,\Vert\hspace{.05em}\zFundtrf\Vert_{\Hk(\M)} + \Vert\hspace{.05em}\zFundtrf\Vert_{\Lp^\infty(\M)} \le \frac C{\Aradius^{\frac32+\outve}} \labeleq{Regularity_of_surfaces_in_asymptotically_flat_spaces__ineq_k} \end{equation*}
if $\Aradius>\Aradius_1$.\footnote{In fact, we get $\Vert\zFundtrf\Vert_{\Wkp^{1,p}(\M)}\le\nicefrac C{\Aradius^{\frac32+\ve-\frac2p}}$ for any $p\in\interval*1\infty$.} In particular, \cite[Thm~1.1]{DeLellisMueller_OptimalRigidityEstimates} implies that there is a center point $\centerz\in\R^3$ and a function $f\in\Ck^2(\sphere^2;\R)$ such that
\begin{equation*}
	\M = \graph f, \qquad
	\Vert f\Vert_{\Wkp^{2,\infty}(\sphere^2_\Aradius(\centerz\,))} \le C\,\Aradius^{\frac12-\outve},\qquad
	\vert\centerz\,\vert\le\ccenterz\,\Aradius + C\,\Aradius^{1-\eta}. \labeleq{Regularity_of_surfaces_in_asymptotically_flat_spaces__ineq_f} \pagebreak[3]\end{equation*}
\end{proposition}

From now on, we assume that the assumptions of Theorem~\ref{Existence} are satisfied, including \eqref{assumptions_Theorem_1} and \eqref{assumptions_Theorem_2} (for some $\outck$ which we will fix later).
Now, we can rigorously define the interval $\intervalI$.
\begin{notation}[Interval \texorpdfstring\intervalI{\intervalI*}]\label{intervalI}
Let $\ccenterz\in\interval*01$, $\c\ge0$, and $\Hradius_0<\infty$ be constants, let $\intervalI=\Cof{\intervalI}[\ccenterz][\c][\Hradius_0]\subseteq\interval*-1*1$ be an interval, and $\{\Phi<\Hradius\!>:\intervalI\times\sphere^2\to\outM\}_{\Hradius>\Hradius_0}$ be a family of maps satisfying for every $\Hradius>\Hradius_0$:
\begin{enumerate}[label={\normalfont(I-\arabic*)}]
\item $\Phi<\Hradius>\in\Ck^1(\intervalI;\Wkp^{1,p}(\sphere^2;\outM))$ for some $p\in\interval2\infty$, i.\,e.\ $\outx\circ\Phi<\Hradius>$ is continuously differentiable as map from $\intervalI$ to the Banach space $\Wkp^{1,p}(\sphere^2;\R^3):=\{(f_i)_{i=1}^3\;|\;f_i\in\Wkp^{1,p}(\sphere^2)\}${\normalfont;}\footnote{Note that $\Phi<\Hradius>$ can be chosen (at least) continuously differentiable as map from $\intervalI\times\sphere^2$ to $\outM$, but this will not matter in the following.} \label{I_psi_C1}
\item $0\in\intervalI$ and $\Phi<\Hradius>(0,\cdot)$ is continuously differentiable; \label{I_differentiable}
\item $\partial*_\gewicht\,(\Phi<\Hradius>)$ is orthogonal to $\M<\Hradius>[\gewicht\,]:=\Phi<\Hradius>(\gewicht,\sphere^2)${\normalfont;}
\item $\M<\Hradius>[\gewicht\,]$ has constant $\gewicht$-weighted expansion, i.\,e.\ $\H<\Hradius>[\gewicht]{+}\gewicht\,\tr<\Hradius>[\gewicht\,]\outzFund\equiv\nicefrac{{-}2}\Hradius$ for every $\gewicht\in\intervalI${\normalfont;} \label{I_CwER}
\item $\M<\Hradius>[\gewicht\,]\in\mathcal A^{\ve,\ve}(\ccenterz,\c)$ for every $\gewicht\in\intervalI${\normalfont;}\label{J_assumptions}
\item $\intervalI$ is maximal, i.\,e.\ if the assumptions~\ref{I_psi_C1}--\ref{J_assumptions} hold for all $\Hradius>\Hradius_0$, an interval $\intervalI'\subseteq\interval*-1*1$, and maps $\{\Phi<\Hradius>':\intervalI'\times\sphere^2\to\outM\}_{\Hradius>\Hradius_0}$, then $\intervalI'\subseteq\intervalI$.\label{I_maximal}
\end{enumerate}
The metric and derived quantities of such a sphere $\M<\Hradius>[\gewicht\,]$ are denoted by $\g<\Hradius>[\gewicht\,]$ etc.\pagebreak[1]
\end{notation}
In particular, \intervalI is (for sufficiently large $\c$ and $\Hradius_0$) non-empty as $0\in\intervalI$ due to Theorem~\ref{existence_of_CMC_leaf} and $\Phi<\Hradius>(0,\sphere^2)$ is a CMC-surface from Theorem~\ref{existence_of_CMC_leaf}. We note that $\Phi<\Hradius>$ is a priori not uniquely defined, but its \lq start value\rq\ $\Phi<\Hradius>(0,\sphere^2)$ is uniquely determined due to Theorem~\ref{CMC_Uniqueness} -- see Lemma~\ref{I_open} for uniqueness of $\Phi$. Furthermore, $\intervalI$ depends on the choice of $\Hradius_0$, $\c$, and $\ccenterz$. In the following, we suppress this dependency and the index $\Hradius$. Additionally, we will always assume that $\Hradius>\Hradius_0$, where $\Hradius_0=\Cof{\Hradius_0}[\mass][{\outve}][\oc][\ccenterz][\c]$ is assumed to be \lq sufficiently\rq\ large. We will choose $\Hradius_0$, $\c$, and $\ccenterz$ after Lemma~\ref{I_open}. \pagebreak[3]\smallskip

As explained in the introduction, we use the same proof structure as Metzger \cite{metzger2007foliations}, i.\,e.\ prove that $\intervalI$ is open by using the implicit function theorem on the map
\[
 (\textbf H + \boldsymbol\cdot\,\textbf P)^{\nu}
	: \interval*-1*1 \times \Wkp^{2,p}(\M) \to \Lp^p(\M)
	: (\gewicht,f) \mapsto (\textbf H + \gewicht\,\textbf P)^{\nu}(\graphnu\nu f).
\]
We note that $\jacobipm{\gewicht_0}<\gewicht_0>$ is the Fr\'echet derivative of this map in the second component at $(\gewicht_0,0)$, if $p>2$. If $\jacobipm$ is invertible, we can thus use the implicit function theorem to extend $\psi$ to a neighborhood of $\intervalI$ such that assumptions~\ref{I_psi_C1}-\ref{I_CwER} are satisfied. Hence, we prove that this pseudo stability operator is invertible. This proof is analog to the one of \cite[Lemma~2.5, Prop.~2.7]{nerz2014CMCfoliation}, but we repeat it nevertheless for readers convenience.
\begin{proposition}[Pseudo stability operator is invertible]\label{stability_operator_invertible}
There are constants ${\outck}_0=\Cof{{\outck}_0}[\mass][\outve][\oc]>0$, $\ccenterz_0=\Cof{\ccenterz_0}[\mass][\outve][\oc][\c]>0$, and $\Hradius_0=\Cof{\Hradius_0}[\mass][\outve][\oc][\c]$ such that the $\gewicht$-weighted pseudo stability operator $\jacobipm{\gewicht}[\gewicht]$ of $\M[\gewicht\,]$ is invertible for every $\gewicht\in\intervalI$ if $\outck\le{{\outck}_0}$, $\ccenterz\le\ccenterz_0$, and $\Hradius>\Hradius_0$. In this case, there exists a constant $C=\Cof[\mass][\outve][\oc]$ such that
\begin{align*}
 \vert\int\jacobipm{\gewicht}\transg\,\transh\d\mug - \frac{6\,\mHaw}{\Hradius^3}\int\transg\,\transh\d\mug\vert \le{}& \frac D{\Hradius^3}\,\Vert\transg\Vert_{\Lp^2(\M)}\,\Vert\transh\Vert_{\Lp^2(\M)},\labeleq{stability_operator_eigenvalue_T}\\
 \frac1{\Hradius^2}\,\Vert h - \transh\Vert_{\Lp^2(\M)} \le{}& \Vert\jacobipm*(h - \transh)\Vert_{\Lp^2(\M)}, \labeleq{stability_operator_eigenvalue_bot} \\
 \frac{6\,\vert\mHaw\vert-D}{\Hradius^3}\,\Vert h\Vert_{\Lp^2(\M)} \le{}& \Vert\jacobipm h\Vert_{\Lp^2(\M)} \labeleq{stability_operator_eigenvalue}
\end{align*}
holds for all functions $g,h\in\Hk^2(\M)$, where $\mHaw=\mHaw(\M)$ denotes the Hawking mass of $\M$ and $D:=C(\ccenterz+\outck+\Hradius^{{-}\outve})$. Here, the \emph{translative part} $\transh$ of any function $h\in\Lp^2(\M)$ is defined as the $\Lp^2(\M)$-orthogonal project of $h$ onto the linear span of the Eigenfunctions $\eflap_i$ of the \normalbrace{negative} Laplace operator for which the corresponding eigenvalues $\ewlap_i$ satisfy $\vert\ewlap_i-\nicefrac2{\Hradius^2}\vert\le\nicefrac1{\Hradius^2}$.\pagebreak[2]
\end{proposition}
We note that we characterized the mass $\mass$ by the limit of the Hawking masses of the Euclidean spheres $\sphere^2_\rradius(0)$. This implies that the Hawking mass of a (sufficiently large) Euclidean sphere $\sphere^2_\rradius(0)$ with respect to the surrounding metric $\outg*$ is non-vanishing. We see that this implies that any surfaces satisfying the assumptions of Proposition~\ref{Regularity_of_surfaces_in_asymptotically_flat_spaces} (for sufficiently large $\Hradius$) has non-vanishing Hawking mass. This is explained in more detail for example in \cite[Appendix~B]{nerz2014CMCfoliation}.
\begin{proof}[Proof of Proposition~\ref{stability_operator_invertible}]
We suppress the index $\gewicht\in\intervalI$ and write $D$ for any constant as in the claim of the proposition. By Proposition~\ref{Regularity_of_surfaces_in_asymptotically_flat_spaces}, there exists a function $f:\sphere^2_\Hradius(\centerz\,)\to\M$ such that
\[ \M = \graph f, \qquad \Vert f\Vert_{\Hk^3(\sphere^2_\Hradius(\centerz\,))} \le C\,\Hradius^{\frac32-\outve}, \]
where we can assume that $\centerz$ is the \emph{Euclidean coordinate center} defined by 
\[ \centerz^i := \fint_{\M}\outx^i\,\d\mug \]
and that it satisfies $\vert\centerz\,\vert\le \ccenterz\,\Hradius + C\,\Hradius^{1-\outve}$. In particular, the eigenvalues of the (negative) Laplace operator $\ewlap_i$ ($\ewlap_i\le\ewlap_{i+1}$) satisfy
\begin{equation*} \vert\ewlap_i-\frac2{\Hradius^2}\vert \le \frac C{\Hradius^{\frac52+\outve}}\quad\forall\,i\in\{1,2,3\}, \qquad\ewlap_j\ge\frac5{\Hradius^2} \quad\forall\,j>3 \labeleq{stability_operator_eigenvalue__ewlap} \end{equation*}
and the corresponding orthogonal eigenfunctions $\eflap_i$ satisfy
\begin{equation*} \Vert\Hesstrf\,\eflap_i\Vert_{\Lp^2(\M)} \le \frac C{\Hradius^{\frac52+\outve}}, \quad\Vert\levi*\eflap_i - \frac{X_i-f_i\,\nu}\Hradius \Vert_{\Lp^2(\M)} \le \frac C{\Hradius^{\frac32+\outve}} \qquad\forall\,i\in\{1,2,3\}, \labeleq{stability_operator_eigenvalue__levi_eflap_i} \end{equation*}
where $X_i\in\R^3$ is a constant vector field (depending on $i\in\{1,2,3\}$ and $\Hradius$) satisfying
\begin{equation*} X_i\cdot X_j = \delta_{ij}\,\Vert f_i\Vert_{\Lp^\infty(\M)}^2,\quad\Vert\outg*(X_i,\nu)-f_i\Vert_{\Lp^2(\M)}\le\frac C{\Hradius^{\frac12+\outve}}\qquad\forall\,i,j\in\{1,2,3\}. \labeleq{stability_operator_eigenvalue__levi_eflap_i__Xi} \end{equation*}
By the Bochner-Lichnerowicz Formel, we know
\[ \frac{\laplace\g*(\levi*\eflap_i,\levi*\eflap_j)}2 = \trtr{\Hesstrf\,\eflap_i}{\Hesstrf\,\eflap_j} + \frac{\ewlap_i\,\ewlap_j}2\,\eflap_i\,\eflap_j - \frac{\ewlap_i+\ewlap_j-\scalar}2\,\g*(\levi*\eflap_i,\levi*\eflap_j), \]
where we used $2\,\ric*=\scalar\,\g$ as $\M$ is two-dimensional. Hence, we get by integration and integration by parts
\[ \vert\frac{\ewlap_i^2}2\,\delta_{ij} - \int\frac\scalar2\g*(\levi*\eflap_i,\levi*\eflap_j)\d\mug \vert \le \frac C{\Hradius^{5+\ve}}\qquad\forall\,i,j\in\{1,2,3\}. \]
Plugging in the (pointwise) assumption on $\outzFund*$ as well as $\H+\gewicht\,\troutzFund\equiv\nicefrac{{-}2}\Hradius$, we conclude using the Gau\ss\ equation
\begin{align*}
 \left|\ewlap_i^2\,\delta_{ij} - \int(\outsc-2\outric*(\nu,\nu))\,\g*(\levi*\eflap_i,\levi*\eflap_j)\d\mug\right. \qquad\quad \\
		- \left. \int (\frac2{\Hradius^2}-\frac{2\gewicht}\Hradius\troutzFund) \g*(\levi*\eflap_i,\levi*\eflap_j)\d\mug \right| \le{}& \frac C{\Hradius^{5+\ve}}.
\end{align*}
Thus, \eqref{stability_operator_eigenvalue__levi_eflap_i} and \eqref{stability_operator_eigenvalue__levi_eflap_i__Xi} imply
\[ \vert\ewlap_i\,(\ewlap_i-\frac2{\Hradius^2})\,\delta_{ij} - \int(\outsc-2\outric*(\nu,\nu)+\frac{2\gewicht}\Hradius\troutzFund)\,\frac{\Vert\eflap_i\Vert_{\Lp^\infty(\M)}^2\,\delta_{ij}-\eflap_i\,\eflap_j}{\Hradius^2}\d\mug\vert \le \frac C{\Hradius^{5+\ve}}. \]
We know
\[ \vert \mHaw - \frac\Hradius{16\pi}\int\outsc-2\outric*(\nu,\nu) \d\mug \vert \le \frac C{\Hradius^\ve} \]
due to the Gau\ss-Bonnet theorem, the Gau\ss\ equation, and the inequalities on $\zFundtrf*$ proven in Proposition~\ref{Regularity_of_surfaces_in_asymptotically_flat_spaces}. Comparing $\eflap_i$ with its analog on the Euclidean sphere, we see
\[ \vert\Vert\eflap_i\Vert_{\Lp^\infty(\M)}^2-\frac3{4\pi\Hradius^2}\vert\le \frac C{\Hradius^{2+\outve}} \]
and we therefore get
\[ \vert(\ewlap_i\,(\ewlap_i-\frac2{\Hradius^2})-\frac{12\,\mHaw}{\Hradius^5})\,\delta_{ij} + \int(\outsc-2\outric*(\nu,\nu)+\frac{2\gewicht}\Hradius\troutzFund)\,\frac{\eflap_i\,\eflap_j}{\Hradius^2}\d\mug\vert \le \frac D{\Hradius^5}, \]
where we used the last integral inequality in \eqref{assumptions_Theorem_1}. By solving this inequality for $\ewlap_i$ and keeping $\ewlap_i\approx\nicefrac2{\Hradius^2}$ in mind, we see
\[ \vert \ewlap_i - \frac2{\Hradius^2} - \frac{6\mHaw}{\Hradius^3} + \int(\frac\outsc2-\outric*(\nu,\nu)+\frac{\gewicht\,\troutzFund}\Hradius)\,\eflap_i^2\d\mug \vert \le \frac D{\Hradius^3} \qquad\forall\,i\in\{1,2,3\} \]
and
\[ \vert \int(\outsc-2\outric*(\nu,\nu)+\frac{2\gewicht}\Hradius\troutzFund)\,\eflap_i\,\eflap_j\d\mug \vert \le \frac D{\Hradius^3} \qquad\forall\,i\neq j\in\{1,2,3\}. \]
Thus, \eqref{jacobipm} implies for $i\neq j\in\{1,2,3\}$
\begin{align*}
 \vert\int \jacobipm*\eflap_i\,\eflap_j \d\mug \vert
	\le{}& \vert\int(\outric*(\nu,\nu)+\gewicht(\frac{\troutzFund}\Hradius + \frac2\Hradius\,\outzFund_{\nu\nu}))\,\eflap_i\,\eflap_j \d\mug\vert \\
				& + \vert\gewicht\int\outzFundnu(\levi*\eflap_i)\eflap_j - \outzFundnu(\levi*\eflap_j)\eflap_i \d\mug\vert + \frac C{\Hradius^{3+\outve}} \\
	\le{}& \vert\frac{2\gewicht}\Hradius\int\outH\,\eflap_i\,\eflap_j\d\mug\vert
					+ \vert\frac\gewicht\Hradius\int\outzFund(\nu,X_i)\eflap_j - \outzFund(\nu,X_j)\eflap_i \d\mug\vert + \frac D{\Hradius^3}.
\end{align*}
Now, we want to plug in the assumptions \eqref{assumptions_Theorem_1} and  \eqref{assumptions_Theorem_2}. But, they are formulated on the Euclidean coordinate spheres and the Euclidean normals $\nicefrac{\outx_i}\rad$ instead of $\M[\gewicht]$ and $\eflap_i$. However, we can nevertheless use these assumptions: By Proposition~\ref{Regularity_of_surfaces_in_asymptotically_flat_spaces} and using the decay assumptions on the derivative of $\outzFund$, the inequalities in \eqref{assumptions_Theorem_1} and \eqref{assumptions_Theorem_2} are also satisfied for $\M$, $D$, and $\nu_i$ instead of $\sphere^2_\Hradius(0)$, $\outck$, and $\nicefrac{\outx_i}\rad$, respectively. Furthermore, we can replace $\nu_i$ by $\eflap_i$. To see this, we first note that in the model case $(\M,\g*)=(\sphere^2,\sphg<\Hradius>*)$, where $\sphg<\Hradius>$ is the standard metric on the Euclidean sphere with radius $\Hradius$, we could choose $\eflap_i=\nicefrac{\nu_i}{\Vert\nu_i\Vert_{\Lp^2(\M)}}=\sqrt{\nicefrac3{4\pi\Hradius^2}}\,\nu_i$. By Proposition~\ref{Regularity_of_surfaces_in_asymptotically_flat_spaces}, this implies
the comparability of $\{\nu_i\}_{i=1}^3$ and $\{\eflap_i\}_{i=1}^3$, i.\,e.
\begin{equation*} \Vert\eflap_i - \sum_{j=1}^3 \nu_j'\,\int\eflap_i\,\nu_j'\d\mug\Vert_{\Lp^2(\M)} \le \frac C{\Hradius^\ve}, \qquad
	\Vert\nu_i' - \sum_{j=1}^3 \eflap_j\,\int\nu_i'\,\eflap_j\d\mug\Vert_{\Lp^2(\M)} \le \frac C{\Hradius^\ve},\labeleq{comparability_nu_f} \end{equation*}
where $\nu_i':=\nicefrac{\nu_i}{\Vert\nu_i\Vert_{\Lp^2(\M)}}$. Thus, we get
\begin{align*}
 \vert\int_{\M<\Hradius>[\gewicht]}\outzFund*(\nu,\eflap_j\,X_i-\eflap_i\,X_j)\d\mug\vert \le{}& \frac D{\Hradius^2}, &
 \vert\int_{\M<\Hradius>[\gewicht]}\troutzFund \d\mug\vert \le{}& D, \label{assumptions_Theorem_1_f}\tag{\ref{assumptions_Theorem_1}{\normalfont$\!\,_\eflapsymbol$}} \\
 \vert\int_{\M<\Hradius>[\gewicht]}\outH\,\eflap_i\,\eflap_j \d\mug\vert \le \frac D{\Hradius^2}, \quad
 \vert\int_{\M<\Hradius>[\gewicht]}\outH\,\eflap_i\d\mug\vert \le{}& \frac D\Hradius, &
 \vert\int_{\M<\Hradius>[\gewicht]}\troutzFund\,\eflap_i\d\mug\vert \le{}& \frac D\Hradius.
\label{assumptions_Theorem_2_f}\tag{\ref{assumptions_Theorem_2}{\normalfont$\!\,_\eflapsymbol$}}
\end{align*}
Using the first inequalities in (\ref{assumptions_Theorem_1_f}) and (\ref{assumptions_Theorem_2_f}), we get \eqref{stability_operator_eigenvalue_T} for $g=\eflap_i$ and $h=\eflap_j$ with $i\neq j\in\{1,2,3\}$. By the corresponding calculation for $i=j\in\{1,2,3\}$, \eqref{stability_operator_eigenvalue_T} holds for $g=h=\eflap_i$ with $i\in\{1,2,3\}$. As it is sufficient to prove \eqref{stability_operator_eigenvalue_T} for $g=\transg$ and $h=\transh$, this proves \eqref{stability_operator_eigenvalue_T} for every $g,h\in\Lp^2(\M)$.

Furthermore, we know
\[ \Vert\jacobipm*g - \laplace g - \frac2{\Hradius^2}g \Vert_{\Lp^p(\M)} \le \frac C{\Hradius^{\frac52+\outve}} \Vert g \Vert_{\Wkp^{1,p}(\M)} \qquad\forall\,p\in\interval*1*\infty,\,g\in\Wkp^{2,p}(\M) \]
and \eqref{stability_operator_eigenvalue__ewlap} therefore implies
\[ \Vert\deform g\Vert_{\Lp^2(\M)} \le \frac{2\Hradius^2}3 \Vert\jacobipm*\deform g\Vert_{\Lp^2(\M)} \qquad\forall\,g\in\Hk^2(\M),\,\deform g:=g-\transg. \]
In particular, \eqref{stability_operator_eigenvalue_bot} and \eqref{stability_operator_eigenvalue} are true for any $g\in\Hk^2(\M)$ with $\transg=0$.

We get for every $g\in\Hk^2(\M)$ with $\Vert\deform g\Vert_{\Lp^2(\M)}^2 \ge \nicefrac1{\Hradius^{\frac12+\outve}}\;\Vert g\Vert_{\Lp^2(\M)}^2$, where $\deform g:=g-\transg$, \begin{align*}
 \vert\int\jacobipm g\,g\d\mug \vert
	\ge{}& \frac{{-}6m-D}{\Hradius^3}\Vert\transg\Vert_{\Lp^2(\M)}^2 + \frac3{2\Hradius^2}\Vert\deform g\Vert_{\Lp^2(\M)}^2 - \frac C{\Hradius^{\frac52+\outve}}\Vert\transg\Vert_{\Lp^2(\M)}\,\Vert\deform g\Vert_{\Lp^2(\M)} \\
	\ge{}& \frac{{-}6m-D}{\Hradius^3}\Vert\transg\Vert_{\Lp^2(\M)}^2 + \frac1{\Hradius^2}\Vert\deform g\Vert_{\Lp^2(\M)}^2
	\ge \frac1{2\Hradius^2}\Vert g\Vert_{\Lp^2(\M)}^2,
\end{align*}
i.\,e.\ \eqref{stability_operator_eigenvalue} is satisfied for each function $g\in\Hk^2(\M)$ with $\Vert\deform g\Vert_{\Lp^2(\M)}^2 \ge \nicefrac1{\Hradius^{\frac12+\outve}}\;\Vert g\Vert_{\Lp^2(\M)}^2$. On the other hand, if $\Vert\deform g\Vert_{\Lp^2(\M)}^2 \le \nicefrac1{\Hradius^{\frac12+\outve}}\;\Vert g\Vert_{\Lp^2(\M)}^2$, then the regularity of the Laplace operator implies
\begin{align*}\hspace{3em}&\hspace{-3em}
 \Vert\jacobipm g\Vert_{\Lp^2(\M)}\,\Vert g\Vert_{\Lp^2(\M)} \\
	\ge{}& \vert\int\jacobipm g\,g\d\mug \vert \\
	\ge{}& \frac{6m-D}{\Hradius^3}\Vert\transg\Vert_{\Lp^2(M)}^2 - \frac C{\Hradius^{\frac52+\outve}}\,\Vert\deform g\Vert_{\Lp^2(\M)}\,\Vert\transg\Vert_{\Lp^2(\M)} - \frac C{\Hradius^2}\Vert\deform g\Vert_{\Hk^2(\M)}\,\Vert\transg\Vert_{\Lp^2(\M)} \\
	\ge{}& \frac{6m-D}{\Hradius^3}\Vert\transg\Vert_{\Lp^2(M)}^2 - \frac C{\Hradius^{\frac52+\outve}} (\Hradius^2\,\Vert\jacobipm\,\deform g\Vert_{\Lp^2(\M)} + \Vert\deform g\Vert_{\Lp^2(\M)})\,\Vert g\Vert_{\Lp^2(\M)}
\end{align*}
and therefore \eqref{stability_operator_eigenvalue} is true for these functions, too. Thus, \eqref{stability_operator_eigenvalue} is proven. Because $\jacobipm*$ is an elliptic operator, this proves all claims of this proposition.
\end{proof}

Using the implicit function theorem, we now deduce that $\psi$ can be extended on a larger interval.
\begin{lemma}[\texorpdfstring{$\Phi$}{The Deformation} is extendable on a neighborhood of \texorpdfstring\intervalI I]\label{Psi_extendable}
Assume that $\outck\le{\outck}_0$, $\ccenterz\le\ccenterz_0$ and $\Hradius>\Hradius_0$ is satisfied for the constants ${\outck}_0=\Cof{{\outck}_0}[\mass][\outve][\oc]$, $\ccenterz_0=\Cof{\ccenterz_0}[\mass][\outve][\oc][\c][\cso][\cdelta]$, and $\Hradius_0=\Cof{\Hradius_0}[\mass][\outve][\oc][\ckappa][\c][\cso][\cdelta]$ of Proposition~\ref{stability_operator_invertible}. For any $\gewicht\in\intervalI$ with $\vert\gewicht\vert<1$ there is a $\kappa>0$ and a $\Ck^1$ map $\Psi:I\cup\interval{\gewicht-\kappa}{\gewicht+\kappa}\to\outM$ which satisfies assumptions~\ref{I_psi_C1}--\ref{I_CwER} of Notation~\ref{intervalI}. Furthermore, all maps $\Phi'$ satisfying the assumptions~\ref{I_psi_C1}--\ref{I_maximal} and $\Phi'(0,\cdot)=\Phi(0,\dot)$ coincide, i.\,e.\ $\Phi$ is uniquely determined by $\Phi(0,\cdot):\sphere^2\to\M<\Hradius>[0]$.\pagebreak[1]
\end{lemma}
\begin{proof}
Let $\gewicht_0\in\intervalI$ with $\vert\gewicht_0\vert<1$ and $p>2$ be arbitrary and suppress the index $\gewicht_0$. The operator $\jacobipm{\gewicht_0}[\gewicht_0]:\Wkp^{2,p}(\M)\to\Lp^p(\M)$ is the Fr\'echet derivative of
\[
 (\textbf H + \boldsymbol\cdot\,\textbf P)^{\nu}
	: \interval*-1*1 \times \Wkp^{2,p}(\M) \to \Lp^p(\M)
	: (\gewicht,f) \mapsto \tensor{(\textbf H + \gewicht\,\textbf P)^{\nu}}[\gewicht_0](f)
\]
with respect to the second component in $(\gewicht_0,0)$. In particular, this linearization is well-defined and invertible due to Proposition~\ref{stability_operator_invertible}. By the implicit function theorem, there is a constant $\kappa>0$ and a $\mathcal C^1$-map $\gamma:\interval{\gewicht_0-\kappa}{\gewicht_0+\kappa}\to\Wkp^{2,p}(\M)$ such that $(\textbf H + \boldsymbol\cdot\,\textbf P)^{\nu}(\gewicht,\gamma(\gewicht))\equiv(\textbf H + \boldsymbol\cdot\,\textbf P)^{\nu}(\gewicht_0,0)$ for any $\gewicht\in\interval{\gewicht_0-\kappa}{\gewicht_0+\kappa}$ and this map is unique within a neighborhood of $0\in\Wkp^{2,p}(\M)$. This implies that $\Phi$ can be $\Ck^1$-extended to $\intervalI\cup\interval{\gewicht_0-\kappa}{\gewicht_0+\kappa}$ and is uniquely defined by $\Phi[\gewicht_0]$. In particular, we conclude by the continuity of $\Phi$ that $\Phi$ is uniquely defined on $\intervalI$ by $\Phi(0,\cdot)$.\pagebreak[3]
\end{proof}

Thus, $\intervalI$ is open if the extension $\Psi$ of $\Phi$ satisfies the regularity assumption~\ref{J_assumptions} of Notation~\ref{intervalI}. Hence, it is sufficient to prove estimates for the derivatives of $\min_{\M[\gewicht\,]}\rad$ and $\volume{\M[\gewicht\,]}$ in order to conclude that $\intervalI$ is open. We will control the second of these derivatives by proving $\psi$ is (in highest order) a pure shift and the first of these derivative by sufficiently bounding the derivative of this shift.
\begin{lemma}[Decay properties of \texorpdfstring{$\rnu$ and $\rnu^\bot$}{the translation- and deformation-part of the lapse function}]\label{decay_rnu_full}
For every $p\in\interval*2\infty$, there are constants $\Hradius_0=\Cof{\Hradius_0}[\mass][\outve][\oc][\c]$, ${\outck}_0=\Cof{{\outck}_0}[\mass][\outve][\oc]>0$, $\ccenterz_0=\Cof{\ccenterz_0}[\mass][\outve][\oc][\c]>0$, and $C=\Cof[\mass][\outve][\oc][\c][p]$ with ${\outck}_0>0$, $\ccenterz_0>0$ and $\Hradius_0<0$ such that
\begin{equation*}
 \Vert \rnu\Vert_{\Wkp^{2,p}(\M)} \le D\,\Hradius^{1+\frac2p}, \qquad
 \Vert \rnu^\bot\Vert_{\Wkp^{2,p}(\M)} \le C\,\Hradius^{\frac12-\ve+\frac2p} \labeleq{decay_rnu}
\end{equation*}
if $\ccenterz\le\ccenterz_0$, $\outck\le{\outck}_0$, and $\Hradius>\Hradius_0$, where $\rnu:=\outg*(\nu,\partial*_\gewicht\Psi)$ denotes the lapse function of $\Psi$ \normalbrace{see Lemma~\ref{Psi_extendable}} and $D:=C(\outck+\ccenterz+\Hradius^{{-}\outve})$, where the index $\gewicht\in\intervalI$ was suppressed. Furthermore, $\Phi$ is in this setting continuously differentiable as map from $\intervalI\times\sphere^2$ to $\outM$. \pagebreak[1]
\end{lemma}
\begin{proof}
Per definition of $\rnu$ and $\Phi$, we know for any $\gewicht\in\intervalI$
\begin{equation*}
 0 \equiv \partial[\gewicht]@{(\H[\gewicht]+\gewicht\;\tr[\gewicht\,]\outzFund)\circ\Phi}
	= \jacobipm[\gewicht]\rnu[\gewicht] + \tr[\gewicht\,]\outzFund,
\end{equation*}
i.\,e.\ $\jacobipm\rnu={-}\tr[\gewicht\,]\outzFund$. This derivative is well-defined on $\intervalI$ by replacing $\Phi$ with its extension $\Psi$ (see Lemma~\ref{I_open}). We conclude
\begin{equation*}
 \Vert \rnu^\bot\Vert_{\Wkp^{2,p}(\M)}
	\le \Hradius^2(\Vert\jacobipm\rnu\Vert_{\Lp^p(\M)}+\Vert\jacobipm(\rnu^T)\Vert_{\Lp^p(\M)})
	\le C\Hradius^{\frac12-\ve+\frac2p} + \frac C{\Hradius^{\frac12+\outve}}\Vert\rnu^T\Vert_{\Lp^p(\M)} \labeleq{decay_rnu_bot_tmp}
\end{equation*}
due to Proposition~\ref{stability_operator_invertible}. Suppressing the index $\gewicht\in\intervalI$, we define
\[ \matrixA_{ij} := \int\jacobipm \eflap_i\;\eflap_j\d\mug \qquad\forall\,i,j\in\{1,2,3\}, \]
where $\{\eflap_i\}_{i\in\N}$ is a complete orthonormal system of $\Lp^2(\M)$ by eigenfunctions of the (negative) Laplace operator with corresponding eigenvalues $\ewlap_i$ ($\ewlap_i\le\ewlap_{i+1}$). We recall that by Proposition~\ref{stability_operator_invertible}
\begin{equation*} \vert\matrixA[\gewicht]_{\oi\oj} + \frac{6m}{\Hradius^3}\eukoutg_{\oi\oj}\vert \le \frac D{\Hradius^3}. \labeleq{estimate_on_A} \end{equation*}
Thus, \eqref{decay_rnu_bot_tmp} implies 
\begin{align*}
 \vert\int\rnu f_\oi\d\mug - \frac{\Hradius^3}{6\mHaw} \int\rnu\,\jacobiext*f_\oi\d\mug\vert
	\le{}& C\,\Vert\rnu^\bot\Vert_{\Lp^2(\M)} + D\,\Vert\rnu^T\Vert_{\Lp^2(\M)} \\
	\le{}& \Hradius^{1-\outve} + D\,\Vert\rnu^T\Vert_{\Lp^2(\M)}
\end{align*}
On the other hand, the identity \eqref{jacobipm} for $\jacobipm*$ and $\jacobipm*\rnu=\troutzFund$ lead to
\[ \int\rnu\;\jacobiext*f_\oi\d\mug
	= {-}\int\troutzFund\,f_\oi\d\mug + 2\int\outzFund*(\nu,\levi*f_\oi)\rnu-\outzFund*(\nu,\levi*\rnu)f_\oi \d\mug.
\]
Taking all together, the (asymptotic) characterization \eqref{stability_operator_eigenvalue__levi_eflap_i} of $\levi*f_\oi$ and $\levi*\trans{\rnu}$ implies
\begin{align*}\hspace{4em}&\hspace{-4em}
 \vert \int\rnu\,f_\oi\d\mug + \frac{\Hradius^3}{6\mHaw}\int\troutzFund\,f_\oi\d\mug - \frac{\Hradius^2}{3m}\sum_{j=1}^3\int\outzFund*(\nu,X_i\,\eflap_j-X_j\,\eflap_i)\d\mug\,\int \rnu\,\eflap_j\d\mug\vert \\
	\le{}& \Hradius^{1-\outve} + D\,\Vert\rnu^T\Vert_{\Lp^2(\M)}.
\end{align*}
Thus, \eqref{decay_rnu} is satisfied due to the first inequality in (\ref{assumptions_Theorem_1_f}) and the third inequality in (\ref{assumptions_Theorem_2_f}).
Finally, we see that $\partial*_\gewicht\Phi=\rnu\,\nu$ and the above regularity of $\rnu$ imply that $\Phi$ is continuously differentiable as map from $\intervalI\times\sphere^2$ to $\outM$. \pagebreak[3]
\end{proof}

As explained above, we can now control the derivatives of the minimal distance from the origin $\min_{\M[\gewicht\,]}\rad$ and the area $\volume{\M[\gewicht\,]}$.
\begin{lemma}[\texorpdfstring{$\gewicht$}b-derivatives of \texorpdfstring{$\min_{\M[\gewicht\,]}\rad$ and $\volume{\M[\gewicht\,]}$}{some quantities}]\label{Derivative_estimates}
There are constants ${{\outck}_0}=\Cof{{\outck}_0}[\mass][\outve][\oc]>0$, ${\ccenterz_0}={\Cof{\ccenterz_0}[\mass][\outve][\oc][\c]}>0$, ${\Hradius_0}={\Cof{\Hradius_0}[\mass][\outve][\oc][\c][\ccenterz]}$, and $C=\Cof[\mass][\outve][\oc][\c]$ such that
\[
 \vert\partial[\gewicht]@{(\rad\circ\varphi)}\vert \le D\,\Hradius, \qquad
 \partial[\gewicht]@{\volume{\M[\gewicht\,]}} \le C\,\Hradius^{\frac32-\outve}
\]
if $\ccenterz\le{\ccenterz_0}$, $\outck\le{\outck}_0$, $\Hradius>\Hradius_0$, and $\gewicht\in\intervalI$, where $D:=C(\outck+\ccenterz+\Hradius^{{-}\outve})$.\pagebreak[1]
\end{lemma}
\begin{proof}
The first inequality holds due to the inequalities \eqref{decay_rnu} of $\rnu$. Further, it is well-known that $\partial*_\gewicht(\d\mug[\gewicht]) = {-}\H\rnu\d\mug$. In particular, the inequalities \eqref{decay_rnu} of $\rnu$ and $\rnu^\bot$ imply
\[ \vert\partial[\gewicht]@{\volume{\M[\gewicht\,]}}\vert = \vert\int\H\rnu\d\mug \vert
		\le C\Vert\rnu^\bot\Vert_{\Lp^2(\M)} + \frac C{\Hradius^{\frac12+\outve}}\Vert\rnu^T\Vert_{\Lp^2(\M)}
		\le C\Hradius^{\frac32-\outve}. \]
Thus, the second inequality holds, too.\pagebreak[3]
\end{proof}
Finally, we can prove that the interval $\intervalI$ is open in $\interval*0*1$.
\begin{lemma}[\texorpdfstring\intervalI I is open]\label{I_open}
Let $\ckappa>0$ be arbitrary. There are positive constants ${\outck}_0=\Cof{{\outck}_0}[\mass][\outve][\oc][\ckappa]$, $\ccenterz_0=\Cof{\ccenterz_0}[\mass][\outve][\oc]\le\ckappa$, $\c_0=\Cof{\c_0}[\mass][\outve][\oc]$, and $\Hradius_0'=\Cof{\Hradius_0'}[\mass][\outve][\oc][\ckappa]$ such that for any $\outck\le{\outck}_0$, $\ccenterz\in\interval*{\ccenterz_0}*\kappa$, $\c\ge\c_0$, and $\Hradius_0>\Hradius_0'$ the interval $\intervalI$ is open in $\interval*-1*1$. In particular, $\M[\gewicht]\in\mathcal A^{\ve,\ve}(\ccenterz_0,\c_0)$ for every $\gewicht\in\intervalI$. \pagebreak[1]
\end{lemma}
\begin{proof}
By Proposition~\ref{Regularity_of_surfaces_in_asymptotically_flat_spaces} and Lemma~\ref{Derivative_estimates}, the estimates on ${\zFundtrf[\gewicht\,]}$ and $\volume{\M[\gewicht\,]}$ only depend on $\min_{\M[\gewicht\,]}\rad$. Let ${\ccenterz[\gewicht]}$ and $\c[\gewicht]$ denote constants such that $\M[\gewicht\,]\in\mathcal A^{\ve,\ve}(\ccenterz[\gewicht],\c[\gewicht])$. By Lemma~\ref{Derivative_estimates}, we can assume
\[ \vert\partial[\gewicht]@{\hspace{.05em}\ccenterz[\gewicht]}\vert \le C(\ccenterz[\gewicht] + \Hradius^{-\ve} + \outck), \]
i.\,e.
\[ \vert\ccenterz[\gewicht]\vert \le C(\outck+\Hradius^{-\ve}+\ccenterz[0]). \]
As $\ccenterz[0\hspace{.05em}]=0$ due to Theorem~\ref{existence_of_CMC_leaf}\footnote{If we look at the alternative assumptions mentioned between the Definitions~\ref{Ck_asymptotically_flat_foliation} and \ref{Not_of-center}, we can choose $\ccenterz[0\hspace{.05em}]>0$ arbitrary small (depending on $\Hradius_0$ but not on $\Hradius>\Hradius_0$). This is also sufficient for this argument.}, we can therefore assume $\ccenterz[\gewicht]\le\ccenterz\le C(\outck+\Hradius^{-\ve})\le\kappa$ if $\outck$ is sufficiently small and $\Hradius$ is sufficiently large. Furthermore, we directly see that $\c[\gewicht]$ can equally be uniformly bounded by some constant $\c_0$. All in all, we get $\M[\gewicht\,]\in\mathcal A^{\ve,\ve}(\ccenterz,\c_0)$ for every $\gewicht$ for which $\Psi(\gewicht,\cdot)$ is well-defined, where $\outck$ and $\c$ do not depend on $\sup\{\vert\gewicht\vert:\gewicht\in\intervalJ\}$. The maximality of $\intervalI$ therefore implies $\Phi=\Psi$ and thus Lemma~\ref{Psi_extendable} ensure that $\intervalI$ is open in $\interval*-1*1$.
\end{proof}
From now on, we assume that $\ccenterz=\ccenterz_0$, $\c=\c_0$, and $\Hradius_0=\Hradius_0'$, where we use the notation of Lemma~\ref{I_open}.

\begin{lemma}[\texorpdfstring\intervalI I is closed]\label{I_closed}
The interval $\intervalI$ is closed in $\interval*-1*1$.\pagebreak[1]
\end{lemma}
\begin{proof}
By the regularity of the lapse function $\rnu$ due to \eqref{decay_rnu} and the regularity of the unit normal -- due to the definition of the second fundamental form and Proposition~\ref{Regularity_of_surfaces_in_asymptotically_flat_spaces} -- we conclude that $\outx\circ\Phi\in\Ck^{0,1}(\intervalI;\Ck^1(\sphere^2;\R^3))$, i.\,e.\ this map is Lipschitz continuous on $\intervalI$ with values in $\Ck^1(\sphere^2;\R^3):=\{(f_i)_{i=1}^3 : f_i\in\Ck^1(\sphere^2)\}$. Thus, we can extend $\Phi$ continuously to a map $\Psi$ on the closed interval $\intervalJ:=\text{closure}(\intervalI)$, i.\,e.\ $\outx\circ\Psi\in\Ck^{0,1}(\intervalJ;\Ck^1(\sphere^2;\R^3))$. Let us assume that $\M[\gewicht]:=\Phi(\gewicht,\sphere^2)\in\mathcal A^{\ve,\ve}(\ccenterz,\c)$ for every $\gewicht\in\intervalJ$ and the constants $\ccenterz$ and $\c$ from Lemma~\ref{I_open} -- we prove this later. Proposition~\ref{stability_operator_invertible} ensures in this case that the pseudo stability operator ${\jacobipm{\gewicht_0}[\gewicht_0]}$ on ${\M[\gewicht_0]}$ is invertible. The same argument as in Lemma~\ref{Psi_extendable} and the uniqueness of $\Phi$ (again due to Lemma~\ref{Psi_extendable}) ensures that $\Psi$ is in fact not only Lipschitz continuously but continuously differentiable, i.\,e.\ $\outx\circ\Ck^1(\intervalJ;\Ck^1(\sphere^2;\R^3))$. The maximality of $\intervalI$ ensures that $\Phi=\Psi$, thus $\intervalI=\text{closure}(\intervalI)$, i.\,e.\ $\intervalI$ is closed.

Left to prove is $\M[\gewicht]\in\mathcal A^{\ve,\ve}(\ccenterz,\c)$, where $\ccenterz$ and $\c$ are as in Lemma~\ref{I_open}. This is a direct implication of Allard's compactness theorem \cite{Allard}. However, we give a more elementary proof for the readers convenience: $\M[\gewicht]:=\Psi(\gewicht,\sphere^2)$ is a $\Ck^1$-submanifold of $\outM$ due to the continuity of $\Psi$ and therefore, the metric $\outg*$ induced a well-defined metric $\g[\gewicht\,]*$ on $\M[\gewicht]$. As we know that $\partial*_\gewicht\,\g[\gewicht\,]*={-}2\,\rnu[\gewicht\,]\,\zFund[\gewicht]*$, the estimates \eqref{decay_rnu} on $\rnu$ and the ones on $\zFund*$ from Proposition~\ref{Regularity_of_surfaces_in_asymptotically_flat_spaces} imply that $\g[\gewicht\,]*\in\Wkp^{1,p}(\sphere)$ depends Lipschitz-continuously on $\gewicht\in\intervalJ$ (for every $p\in\interval2\infty$). Here (and in the following), we suppress the pullback along $\Psi$. Thus, $\tr[\gewicht\,]\hspace{.05em}\outzFund*\in\Wkp^{1,p}(\sphere^2)$ and the second fundamental form $\zFund[\gewicht]*\in\Lp^2(\sphere^2)$ also depend continuously on $\gewicht\in\intervalJ$. This implies that $\H[\gewicht]\equiv\nicefrac{{-}2}\Hradius-\gewicht\,\tr[\gewicht\,]\outzFund*\in\Wkp^{1,p}(\sphere^2)$ does so, too. Hence, we get $\M[\gewicht]\in\mathcal A^{\ve,\ve}(\ccenterz,\c)$ for every $\gewicht\in\intervalJ$. By the above argument, this proves the lemma.
\pagebreak[3]
\end{proof}

\begin{proof}[Proof of the existence Theorem~\ref{Existence} -- without foliation property]
The interval $\intervalI$\linebreak[3] of all weights $\gewicht$ such that there exists a surface ${\M[\gewicht\,]}$ with constant $\gewicht$-weighted expansion ${\H[\gewicht]}+\gewicht\,{\tr[\gewicht\,]\hspace{.05em}}\outzFund*\equiv\nicefrac{{-}2}\Hradius$ is non-empty (Theorem~\ref{existence_of_CMC_leaf}), as well as open (Lemma~\ref{I_open}) and closed (Lemma~\ref{I_closed}) in and thus equal to $\interval*-1*1$. In particular, the surfaces ${\M[\pm]<\Hradius>}:={\M[\pm1]<\Hradius>}$ exist for every $\Hradius>\Hradius_0$.\pagebreak[3]
\end{proof}

\begin{proof}[Proof of the uniqueness Theorem~\ref{Uniqueness}]
Let $\M[\pm]$ be such a CE-surfaces -- where $\pm$ is a fixed sign. We define the interval $\intervalI\subseteq\interval*-1*1$ equally to the one in Notation~\ref{intervalI} replacing the assumption~\ref{I_differentiable} by \lq${\pm}1\in\intervalI$, $\Phi<\Hradius>({\pm}1,\cdot)$ is continuously differentiable, and $\Phi<\Hradius>({\pm}1,\sphere^2)=\M[\pm]$\rq. Repeating the arguments of Section~\ref{existence}, we conclude that $\intervalI$ is open and closed in $\interval*-1*1$. In particular, there is a CMC-surface $\M:=\M[0]$ satisfying the assumptions of Proposition~\ref{Regularity_of_surfaces_in_asymptotically_flat_spaces} and therefore the ones of Theorem~\ref{CMC_Uniqueness}. Thus, $\M$ is a leaf of the CMC-foliation. By the uniqueness of the deformation $\Phi$ due to Lemma~\ref{Psi_extendable}, we conclude that $\M[\pm]$ is the CE-sphere constructed in Theorem~\ref{Existence}.\pagebreak[3]
\end{proof}

To prove that the CE-surfaces foliate the space, we will need to control the (ADM-)\linebreak[1]linear momentum.
\begin{proposition}[Linear momentum is small]\label{linear_Momentum_small}
Let ${\outve}>0$ be a constant and $(\outM,\outg*,\outx,\outzFund*,\outmomden*,\outenden)$ be a $\mathcal C^2_{\frac12+\ve}$-asymptotically flat initial data set. If
\[ \oc_2 \ge \vert\,\int_{\sphere^2_\rradius(0)}\outzFund_{\oj\ok}\frac{\outx_\oi\,\outx^\oj\,\outx^\ok}{\rad^3} \d\mug\,\vert\qquad\forall\,\rradius>R_0 \]
holds for some constants $\oc_2>0$ and $R_0>0$, then 
\begin{equation*} \liminf_{\rradius\to\infty} \int_{\sphere^2_\rradius(0)}\tr\outzFund\,\frac{\outx_i}\rad \d\mug \le \impuls_i \le \limsup_{\rradius\to\infty} \int_{\sphere^2_\rradius(0)}\tr\outzFund\,\frac{\outx_i}\rad \d\mug \labeleq{linear_Momentum_small_ineq}\end{equation*}
and the implication
\[ \vert\int_{\sphere^2_\rradius(0)} \outH\,\frac{\outx_\oi}\rad \d\mug\vert \le \frac{\oc_2}{\rradius^\ve}
	\qquad\Longrightarrow\qquad
		\int_{\sphere^2_\rradius(0)}\tr\outzFund\,\frac{\outx_i}\rad \d\mug
			\xrightarrow{\rradius\to\infty} \impuls_i \]
holds for each $i\in\{1,2,3\}$. Here, $\impuls=(\impuls_1,\impuls_2,\impuls_3)\in\R^3$ denotes the \normalbrace{ADM}** linear momentum defined by
\[ 
	\impuls_\oi :\xleftarrow{\rradius\to\infty} \impuls<\rradius>_\oi := \frac1{8\pi}\sum_{\oj=1}^3\int_{\sphere^2_\rradius(0)} \momentum_{\oi\oj}\frac{\outx_\oj}\rradius\d\eukmug  \quad\text{with}\quad\momentum*:=\outH\,\outg*-\outzFund.
\]
\end{proposition}
Note that we can apply this proposition to every $\mathcal C^2_{\frac12+\outve}$-asymptotically flat initial data set with $\mathcal C^0_2$-asymptotically vanishing second fundamental form.
\begin{proof}
First, we note that for every $\rradius'\ge\rradius>R_0$
\[ \vert\impuls<\rradius>_\oi - \impuls<\rradius'>_\oi\vert
	\le \int_{\R^3\setminus B_{\rradius}(0)}\vert\hspace{.05em}\outdiv(\momentum_\oi)\vert\d\outmug \le \frac C{\rradius^\ve}, \]
where we used the Gau\ss\ theorem and where $\outmug$ denotes the three-dimensional volume measure with respect to $\outg*$. In particular, the linear momentum $\impuls$ is well-defined and $\vert\impuls-\impuls<\rradius>\vert\le\nicefrac C{\rradius^\ve}$. We see that for every $\rradius'\ge\rradius>R_0$
\begin{align*}
 2\,\outc_2
	\ge{}& \vert\int_{\sphere^2_{\rradius'}(0)}(\outH-\troutzFund)\,\frac{\outx_\oi}{\rradius'} \d\mug - \int_{\sphere^2_{\rradius}(0)}(\outH-\troutzFund)\,\frac{\outx_\oi}\rradius \d\mug \vert \\
	\ge{}& \vert\int_{B_{\rradius,\rradius'}} \outzFund(e_\ok,\outlevi[e](\frac{\outx_\oi\outx_\ok}{\rad^2})) \d\Vol\vert - \frac C{\rradius^\ve}
\end{align*}
where $B_{\rradius,\rradius'}:=\{x\in\R^3\,:\,\rradius<\vert x\vert<\rradius'\}$ and where $\Vol$ and $\outlevi[e]$ denotes the measure and the Levi-Civita connection on $B_{\rradius',\rradius}$ with respect to the Euclidean metric $\eukoutg*$. \pagebreak[1]Thus, we conclude by the Fubini-Theorem
\begin{align*}
 2\,\outc_2
	\ge{}& \vert \int_{\rradius}^{\rradius'} \int_{\sphere^2_r(0)} \frac1{r^2}\,\outzFund(e_\ok,\outlevi[e](\outx_\oi\outx^\ok)) - \frac2{r^3}\,\outzFund_{\ok\ol}(\outx_\oi\outx^\ok\outx_\ol)\d\eukmug\d r \vert - \frac C{\rradius^\ve} \\
	\ge{}& \vert \int_{\rradius}^{\rradius'} r^{{-}2} \int_{\sphere^2_r(0)} \outzFund_{\ok\oi} \outx^\ok + \outH \outx_\oi - 2\outzFund_{\nu\nu}\,\outx_\oi\,\d\eukmug\d r \vert - \frac C{\rradius^\ve} \\
	\ge{}& \vert \int_{\rradius}^{\rradius'} \frac 2r (\int_{\sphere^2_r(0)} \troutzFund\,\frac{\outx_\oi}r \,\d\eukmug - \impuls_\oi)\d r \vert - \frac C{\rradius^\ve}.
\end{align*}
If one of the inequalities in \eqref{linear_Momentum_small_ineq} did not hold, then there would be a $\kappa>0$ such that $\int_{\sphere^2_r(0)} \troutzFund\;\nicefrac{\outx_\oi}r \,\d\eukmug - \impuls_\oi>\kappa$ or $\int_{\sphere^2_r(0)} \troutzFund\;\nicefrac{\outx_\oi}r \,\d\eukmug - \impuls_\oi<{-}\kappa$. In both cases, we would get
\[
 2\,\outc_2 \ge \vert \int_{\rradius}^{\rradius'} \frac\kappa r \d r\vert - \frac C{\rradius^\ve} = \kappa (\ln(\rradius') - \ln(\rradius)) - \frac C{\rradius^\ve}
	\xrightarrow{\rradius'\to\infty} \infty \]
contradicting $\outc_2\relax<\infty$. Thus, both inequalities in \eqref{linear_Momentum_small_ineq} hold.

Define $f_i(\rradius):=\int_{\sphere^2_\rradius(0)}(\outH-\tr\outzFund*)\,\nicefrac{\outx_i}\rad\d\mug$. A calculation as the one above implies
\[ f(\rradius') - f(\rradius) + \frac2r \int_\rradius^{\rradius'} (\int_{\sphere^2_r(0)}\troutzFund\,\frac{\outx_i}r \d\mug - \impuls_i )\d r = \int_\rradius^{\rradius'} \text{err}_i(r) \d r, \]
where $\text{err}_i(r)$ is some error term with $\vert\text{err}_i(r)\vert\le\nicefrac C{r^{-1-\outve}}$. Thus, we get
\[ \vert f_i'(\rradius) + \frac{2\,f_i(\rradius)}\rradius - \frac{2\impuls_i}\rradius \vert \le \frac C{\rradius^{1+\outve}} \qquad\forall\,\rradius>R_0. \]
Solving this ordinary differential (asymptotically) equation (see \cite[Prop.~C.1]{nerz2014CMCfoliation} for more information), we conclude $\vert f_i(\rradius) - \impuls_i\vert\le\nicefrac C{\rradius^\ve}$.
\end{proof}

\begin{proof}[Proof of the foliation property in Theorem~\ref{Existence}]
Fix a sign $\pm$ and a (large) constant $\Hradius_1$. Denote by $\M:=\M<\Hradius_1>$ the CE-surfaces with constant expansion $\H\pm\troutzFund\equiv\nicefrac{{-}2}{\Hradius_1}$ which exists due to the already proven part of Theorem~\ref{Existence}. By Proposition~\ref{stability_operator_invertible}, the pseudo stability operator is invertible. Thus, there exists a map $\Phi:\interval{\Hradius_1-\ve}{\Hradius_1+\ve}\times\sphere^2\to\outM$ such that for every $\Hradius\in\interval{\Hradius_1-\ve}{\Hradius_1+\ve}$ the surface $\M<\Hradius>:=\Phi(\Hradius,\sphere^2)$ is a CE-surface with $\H\pm\troutzFund=\nicefrac{{-}2}\Hradius$. We see that the lapse function $\rnu:=\outg*(\nu[\Hradius],\partial*_\Hradius\Phi)$ satisfies
\[ \jacobipm{\pm1}\rnu \equiv \partial[\Hradius](\frac{{-}2}\Hradius) \equiv \frac2{\Hradius^2}. \]
Therefore, we conclude for $\rnu':=\rnu-1$ and $\deform{\rnu'}:=\rnu-1-\trans{\rnu}$
\[
 \Vert {\rnu'}^\bot\Vert_{\Hk^2(\M)} \le C\,\Hradius^{\frac12-\outve} \qquad\text{due to}\qquad
 \vert\hspace{.05em}\jacobipm{\pm1}(\rnu') \vert \le \frac C{\Hradius^{\frac52+\outve}}.
\]
Repeating the arguments of Lemma~\ref{decay_rnu_full}, we conclude using Proposition~\ref{linear_Momentum_small}
\begin{align*}\labeleq{foliation_rnu_greater_zero}
 \Vert\rnu'\Vert_{\Hk^2(\M)}
  \le{}& C\,\Hradius\vert\int(\div\outzFundnu+\frac{{-}3\troutzFund+2\outH}\Hradius - \frac2\Hradius\outzFund(\nu,\vecoff{{\rnu'}^T}))\nu_\oi + \frac2\Hradius\outzFund_{\nu\oi}\rnu'\d\mug\vert \\
		& + C\,\Vert\rnu'\Vert_{\Lp^2(\M)} \\
	\le{}& D\,(\Hradius+\Vert\rnu'\Vert_{\Lp^2(\M)}),
\end{align*}
where $D:=C(\outck+\Hradius^{{-}\outve})$ and $C=\Cof[\mass][\outve][\oc]$. We conclude that $\M<\Hradius>$ satisfies the assumptions of Proposition~\ref{Regularity_of_surfaces_in_asymptotically_flat_spaces} by repeating the argument of Lemma~\ref{Derivative_estimates}. Using Theorem~\ref{Uniqueness}, we see that $\M<\Hradius>$ is the surface constructed in Theorem~\ref{Existence}. As $\Hradius$ was arbitrary (sufficiently large), we can assume that $\Phi:\interval{\Hradius_1}\infty\times\sphere^2\to\outM$ satisfies the assumption made above. By \eqref{foliation_rnu_greater_zero}, we know that $\rnu$ is (strictly) positive for sufficiently large $\Hradius$ and sufficiently small $\outck$. In particular, $\Phi$ is a foliation.\pagebreak[3]
\end{proof}
\begin{proof}[Proof of the time-regularity Theorem~\ref{Regularity_over_time}]
Denote by $\tensor\Phi[t,\pm]:\interval{\Hradius_1}\infty\times\sphere\to\outM[t]$ the CE-foliation due to Theorem~\ref{Existence}. Note that $\Hradius_1$ can be chosen independently of $t$ as the temporal foliation is assumed to be uniformly $\Ck^2_{\frac12+\outve}$-asymptotically flat. Now\vspace{-.4em}, we define the maps ${\tensor\Phi[\pm]}:I\times\interval{\Hradius_1}\infty\times\sphere\to\uniM:(t,\Hradius,p)\mapsto{\tensor\Phi[t,\pm]}(\Hradius,p)$. We suppress the sign index $\pm$ in the following.

Fix a time $t_0\in I$ and a $\Hradius>\Hradius_1$ and suppress the corresponding indexes. We see that $\jacobipm{\pm}$ is the linearization of the map
\[ \textbf H \pm\textbf P : I\times\Wkp^{2,p}(\M) \to \Lp^p(\M) : (t,f) \mapsto (\H[t] \pm \troutzFund[t])(\graphnu\nu_t f) \]
in the second component at $(t_0,0)$.\pagebreak[1] Here, $(\H[t] \pm \troutzFund[t])(\graphnu\nu_t f)$ denotes the expansion of the graph of $f$ at time $t\in I$, where
\[ \graphnu\nu_t f := \{\unisymbol x^{-1}(t,\outx[t_0](p)+f\,D_{\nu}\outx[t_0])\in\outM[t] \,:\,p\in\M\} \]
By Proposition~\ref{stability_operator_invertible}, we know that $\jacobipm{\pm}$ is invertible, thus the implicit function theorem ensures that there is a curve $f:\interval{t_0-\ve}{t_0+\ve}\to\Hk^2(\M)$ such that $(\textbf H\pm \textbf P)(t,f(t))\equiv(\textbf H\pm \textbf P)(0,0)$. By the uniqueness Theorem~\ref{Uniqueness}, this implies $\graphnu\nu_t(f(t))=\M<\Hradius>[t]$. Thus, we can choose $\Phi$ to be smooth (at least $\Ck^1$).
\end{proof}

\section{Invariance in time}\label{invariance}
In this section, we prove that the unique CE-surfaces constructed in Theorem~\ref{Existence} are asymptotically independent of time if the linear momentum vanishes. As explained, this is to be expected, as the CMC spheres asymptotically evolve in time as given by the fraction of (ADM) linear momentum and (ADM) mass \cite[Theorem 4.1]{nerz2013timeevolutionofCMC}, the CE-spheres are asymptotically just shifts of the CMC spheres (due to Section~\ref{existence}), and it is appropriate to assume that the last shift is asymptotically independent of time.
\begin{theorem}[Invariance of the CE-surfaces in time]\label{Invariance}
Assume that the temporal foliation $(\outM[t],\outg[t]*,\outtensor[t] x,\outzFund[t],\outenden[t],\outmomden[t],\ralpha[t])_{t\in I}$ satisfies the assumptions of Theorem~\ref{Regularity_over_time} and let $\uniPhi[\pm]$ denote the foliation as defined in Theorem~\ref{Regularity_over_time}. If
\[ \vert\ralpha[t]-1\vert + \rad\,\vert\outlevi*\ralpha\vert_{\outg*} + \rad^2\,\vert\outHess\,\ralpha\vert_{\outg*} \le \frac{\oc}{\rad^{\frac12}},, \]
then $\uniPhi[\pm]$ is \normalbrace{asymptotically} a shift. If additionally
\[ \vert\int_{\sphere^2_\rradius(0)}\ralpha\,\outx[t]_\oi\d\mug \vert \le \outck\,\rradius^2, \quad
		\vert\int_{\sphere^2_\rradius(0)}\einstein[t]_{kl}\frac{\outx[t]^k\,\outx[t]^l\,\outx[t]_i}{\rad^2}\d\mug\vert \le \outck \qquad\forall\,\rradius> R_0 \]
holds for every $\oi\in\{1,2,3\}$ and if the Einstein Tensor $\einstein*$ of $\uniM*$ satisfies
\[ \Big| \rad[t]\,\einstein[t]_{\ok\ol}\outx[t\,]^\ok\,\outx[t]^\ol\Big| \le \outc, \]
then this shift is small, i.\,e.
\begin{equation*}
 \Vert\eukoutg(\partial[t]@{\uniPhi<\Hradius>},\nu<\Hradius>)\Vert_{\Wkp^{1,\infty}(\M[t,\pm]<\gewicht>)} \le C(\outck+\Hradius^{{-}\outve}), \labeleq{evolution_linear_momentum}
\end{equation*}
where $\nu<\Hradius>[t]$ denotes the outer unit normal of $\M<\Hradius>[t]$, $\uniPhi<\Hradius>:=\uniPhi(\cdot,\Hradius,\cdot)$, $C:=\Cof[\mass][\outve][\oc]$.
\end{theorem}
We directly see, that the descriptive version, Corollary~\ref{Invariance_descriptive}, is a direct corollary of Theorem~\ref{Invariance}. Let $(\outM[t],\outg[t]*,\outtensor[t] x,\outzFund[t],\outenden[t],\outmomden[t],\ralpha[t])_{t\in I}$ and $f$ satisfy the assumptions of Theorem~\ref{Invariance} -- in particular, the ADM linear momentum is small due to Proposition~\ref{linear_Momentum_small}. 
\begin{proposition}[Characterization of the lapse function \texorpdfstring{$\rnu$}\relax]\label{Characterization_evolution}
Assume that the assumptions of Theorem~\ref{Invariance} are satisfied and let $\uniPhi$ denote the foliation as defined in Theorem~\ref{Regularity_over_time}. The lapse function $\rnu<\Hradius>[t]:=\outg[t]*(\nu<\Hradius>[t],\partial*_t\uniPhi)$ is uniquely characterized by
\begin{equation*}
 \jacobipm{\pm}<\Hradius>[t]\,\rnu<\Hradius>[t] = {-} \jacobit<\Hradius>[t]\,\ralpha, \labeleq{rnu_characterization_evolution}
\end{equation*}
where $\nu<\Hradius>[t]$ is the normal of $\M<\Hradius>[t]\hookrightarrow(\outM[t],\outg[t]*)$.\pagebreak[1] There are constants $\Hradius_1=\Cof{\Hradius_1}[\mass][\outve][\oc][\outck]$ and $C=\Cof[\mass][\outve][\oc]$ such that
\begin{equation*}
 \vert \jacobipm{\pm}\rnu - \div\outzFundnu - \frac1\Hradius\troutzFund \mp\laplace\ralpha \mp \outric*(\nu,\nu) \vert \le \frac C{\Hradius^3}(\outck+\Hradius^{-\ve}). \labeleq{jacobipm_evolution}
\end{equation*}
Here, the indices $\Hradius>\Hradius_1$ and $t\in I$ were suppressed.
\end{proposition}
\begin{proof}
Per construction of $\uniPhi$, we know
\[ 0 \equiv \partial[t]@{\H<\Hradius>[t]\pm\tr<\Hradius>[t]\outzFund[t]*} = \jacobipm{\pm}<\Hradius>[t]\,\rnu<\Hradius>[t] + \jacobit<\Hradius>[t]\ralpha[t]. \]
Thus, \eqref{rnu_characterization_evolution} holds and is a unique characterization of $\rnu$ due to the invertibility of $\jacobipm{\pm}<\Hradius>[t]$. Using \eqref{jacobit} and \eqref{rnu_characterization_evolution}, we conclude the inequalities due to the assumptions \eqref{decay_outzFund} and \eqref{decay_lapse} on $\outzFund*$, ${\outmomden}$, and $\ralpha$, as well as the regularity of $\M<\Hradius>$ due to Proposition~\ref{Regularity_of_surfaces_in_asymptotically_flat_spaces}.
\end{proof}
\begin{proof}[Proof of Theorem~\ref{Invariance}]
We choose an arbitrary $t\in I$ and $\Hradius>\Hradius_0$ and suppress the corresponding indices. For the first part, we see that $\rnu$ satisfies
\[
 \Vert \rnu\Vert_{\Wkp^{2,p}(\M)} \le C\,\Hradius^{\frac12+\frac2p-\outve}, \qquad
 \Vert \deform\rnu\Vert_{\Wkp^{2,p}(\M)} \le C\,\Hradius^{\frac2p-\frac12}
\]
due to \eqref{jacobipm_evolution} and the regularity of $\jacobipm{\pm}$ (Proposition~\ref{stability_operator_invertible}), where $p\in\interval*1\infty$ is arbitrary and again $\deform{\rnu}:=\rnu-\trans{\rnu}$. In particular, $\Phi$ is (asymptotically) a shift.

For the second part, we see that the inequality \eqref{evolution_linear_momentum} holds if and only if
\[ \text{err} := \Vert\trans\rnu\Vert_{\Lp^2(\M)} \le \sum_{i=1}^3\vert\int\rnu\,\eflap_i\d\mug\vert \le C(\outck+\Hradius^{{-}\outve})\,\Hradius =: D\,\Hradius, \]
where $\{\eflap_i\}_{i\in\N}$ again is a complete orthonormal system of $\Lp^2(\M)$ of eigenvalues of the (negative) Laplace operator with corresponding eigenvalues $\ewlap_i$ satisfying $\ewlap_i\le\ewlap_{i+1}$. Using the asymptotic characterization of $\jacobipm{\pm}$ on $\Lp^2(\M)^T:=\lin\{\eflap_i\}_{i=1}^3$, \eqref{estimate_on_A}, and inequality \eqref{stability_operator_eigenvalue_T}, we conclude
\begin{align*}
 \text{err}
	\le{}& C\,\Hradius^3\,\sum_{i=1}^3 \vert\int\trans{\rnu}\,\jacobipm{\pm}*\eflap_i\d\mug \vert
	\le C\,\Hradius^3\,\sum_{i=1}^3 \vert\int\rnu\,\jacobipm{\pm}*\eflap_i\d\mug \vert + C\,\Hradius^{\frac12-\outve}\Vert\deform\rnu\Vert_{\Lp^2(\M)} \\
	\le{}&  C\,\Hradius^3\,\sum_{i=1}^3 \vert\int\jacobipm{\pm}*\rnu\,\eflap_i\d\mug - 2\int\outzFund*(\nu,\levi*\trans\rnu)\,f-\outzFund*(\nu,\levi*f_\oi)\,\trans\rnu\d\mug\vert + D\,\Hradius.
\end{align*}
Now, we use the asymptotic identity \eqref{stability_operator_eigenvalue__levi_eflap_i} for $\levi*f^T$ and the characterization of $\rnu$ due to Proposition~\ref{Characterization_evolution} to get
\begin{align*}
 \text{err}
	\le{}& C\,\Hradius^3\,\sum_{i=1}^3 \vert\int\jacobit\ralpha\;\eflap_i\d\mug + \frac2\Hradius\sum_{j=1}^3\int\outzFund*(\nu,X_j)\,f-\outzFund*(\nu,X_i)\,\rnu\d\mug\,\int\rnu\,\eflap_j\d\mug \vert \\
		& + D\,\Hradius + D\,\Vert\trans\rnu\Vert_{\Hk(\M)} \\
	\le{}& C\,\Hradius^3\,\sum_{i=1}^3 \vert\int(\div\outzFundnu + \frac1\Hradius\troutzFund \pm\laplace\ralpha \pm \outric*(\nu,\nu) )\,f_\oi\d\mug \vert + D\,\Hradius + D\,\text{err} \\
	\le{}& C\,\Hradius^3\,\sum_{i=1}^3 \vert\int(\frac1\Hradius\outH \pm \outric*(\nu,\nu) )\eflap_\oi - \frac{\outzFund(\nu,X_i)}\Hradius \mp \ewlap_i\,\ralpha\,\eflap_i\d\mug \vert + D\,\Hradius + D\,\text{err}.
\end{align*}
Now we use the asymptotic antisymmetry of $\eflap_i$, the comparability of $\eflap_i$ and $\nu_i$, the estimates on $\centerz$ and Proposition~\ref{linear_Momentum_small} to conclude
\[ \text{err} \le C\,\Hradius^2\,\sum_{i=1}^3 \vert\int\outric*(\nu,\nu)\,\nu_i\d\mug \vert + D\,\Hradius. \]
However, this last integral is bounded by $C\,\vert\centerz\,\vert\le D\,\Hradius$ due to the Gau\ss\ theorem. This was proven in full detail by the author in \cite[Lemma~A.3]{nerz2014CMCfoliation}.
\end{proof}\bigskip

\bibliography{bib}
\bibliographystyle{alpha}\vfill
\end{document}

%% file: header__arxiv.tex
\usepackage{ifpdf}
\ifpdf
 \usepackage[implicit,naturalnames,hyperfootnotes=false,
						 final]{hyperref}
 \usepackage[pdftex]{color}
\else
 \usepackage{color}																									%
 \gdef\texorpdfstring#1#2{#1}
\fi
\usepackage{lmodern}																								%
\usepackage{xcolor}																									%
\usepackage{enumitem}
\usepackage{amsmath}																								%
\usepackage{amssymb}																								%
\usepackage{amsthm}																									%
\usepackage{bbm}																										%
\usepackage{amsfonts}																								%
\usepackage{mathrsfs}																								%
\usepackage{calligra}																								%
\usepackage{dsfont}																									%
\usepackage{nicefrac}																								%
\usepackage{esint}																									%
\usepackage{mathtools}																							%
\usepackage{ifthen}																									%
\usepackage{xparse}																									%
\usepackage{xspace}																									%
\usepackage{calc}																										%
\usepackage{partial}																%
\usepackage{auto-eqlabel}																						%
\usepackage{bracket-hack}																						%
\usepackage[nostandardtensors]{simpletensors}												%
\NewDocumentCommand\outsymbol{om}{\overline{\IfValueTF{#1}{#1{#2}}{#2}}}
\NewDocumentCommand\unisymbol{om}{\widehat{\IfValueTF{#1}{#1{#2}}{#2}}}
\let\tensor\IndexSymbol
\NewDocumentCommand\outtensor{d()}{\IndexSymbol(\outsymbol[#1])}
\gdef\ccenterz{\specialconstant[\!]{\centerz}}
\gdef\cso{\specialconstant[\hspace{-.05em}]{S}}
\NewDocumentCommand\ckappa{d<>ot^}{\IndexSymbol\kappa}
\NewDocumentCommand\cdelta{d<>ot^}{\IndexSymbol\delta}
\NewDocumentCommand\specialconstant{omd<>ot_}
{\IfBooleanTF{#5}{{\c{[#1]{\c*}}<#3>[#4]_{#2}}_}{\c{[#1]{\c*}}<#3>[#4]_{#2}}}%
\gdef\matrixA{\IndexSymbol[\!]{\textrm A}}
\NewDocumentCommand\vecoff{t_}
{\IfBooleanTF{#1}\vecoffsnd\vecoffthd}
\gdef\vecoffsnd#1{\vecoff{f_{#1}}}
\gdef\vecoffthd#1{\IndexSymbol{X_{#1}}}
\gdef\gewicht{\mathcal b}
\let\Hradius\sigma
\let\Aradius r
\let\rradius R
\NewDocumentCommand\partialt{d()}{\IfValueTF{#1}{(#1)\bolddot{\vphantom]}}{\bolddot}}%
\NewDocumentCommand\partialr{d()}{\IfValueTF{#1}{\partialr{(#1)}}{\bolddash}}%
\newlength\bolddotlength
\gdef\bolddot#1{\setlength\bolddotlength{(\widthof{\ensuremath{\boldsymbol{#1}}}-\widthof{\ensuremath{#1}})/2}%
 \hskip-\bolddotlength
 \boldsymbol{\dot{\phantom{\boldsymbol{#1}}}}
 \hskip-\bolddotlength
 \setlength\bolddotlength{\widthof{\ensuremath{#1}}}
 \hskip-\bolddotlength#1}
\gdef\bolddash#1{#1\boldsymbol{\vphantom{#1}'}}
\newtheoremstyle{mytheorem}{3pt}{}{\itshape}{}{\bfseries}{\nopagebreak\newline}{.5em}{}%
\newtheoremstyle{mydefinition}{3pt}{}{}{}{\bfseries}{\nopagebreak\newline}{.5em}{}%
\gdef\mytheorem{\theoremstyle{mytheorem}}
\gdef\mydefinition{\theoremstyle{mydefinition}}
\gdef\mytheoremcounter{theorem}
\mytheorem
\newtheorem{theorem}{Theorem}[section]
\newtheorem{corollary}[\mytheoremcounter]{Corollary}
\newtheorem{lemma}[\mytheoremcounter]{Lemma}
\newtheorem{proposition}[\mytheoremcounter]{Proposition}
\newtheorem{notation}[\mytheoremcounter]{Notation}
\mydefinition
\newtheorem{definition}[\mytheoremcounter]{Definition}
\theoremstyle{remark}
\newtheorem{remark}[\mytheoremcounter]{Remark}
\gdef\ii{I}\gdef\ij{J}\gdef\ik{K}\gdef\il{L}
\gdef\oi{i}\gdef\oj{j}\gdef\ok{k}\gdef\ol{l}
\NewTensor[\outsymbol]\outx x
\NewTensor[\outsymbol]\mass m
\let\oldPhi\Phi
\RenewTensor\Phi\oldPhi
\NewTensor[\outsymbol]\outPhi\oldPhi
\NewTensor[\unisymbol]\uniPhi\oldPhi
\let\oldphi\phi
\RenewTensor\phi\oldphi
\NewTensor\genus g
\makeatletter
\newdimen\middle@width
\newcommand*\phantomas[3][c]{\ifmmode\makebox[\widthof{$#2$}][#1]{$#3$}\else\makebox[\widthof{#2}][#1]{#3}\fi}%
\NewDocumentCommand\tracefree{m}
{\setlength\middle@width{\widthof{\ensuremath{#1}}}
 #1\hskip-\middle@width
 \phantomas[r]{#1}{\vphantom{\ensuremath{#1}}^{\,\!^\circ}}}
\NewDocumentCommand\mean{m}
{\setlength\middle@width{\widthof{\ensuremath{#1}}}
 \phantomas{#1}{{=\!}}
 \hskip-\middle@width#1}
\makeatother
\NewDocumentCommand\intervalI{s}{\IfBooleanTF{#1}I{\ensuremath{\intervalI*}\xspace}}%
\NewDocumentCommand\intervalJ{s}{\IfBooleanTF{#1}J{\ensuremath{\intervalJ*}\xspace}}%
\gdef\schwarzs{\mathcal S}\let\schws\schwarzs
\gdef\schwarzoutg{\outg[\schwarzs]}
\let\schwarzoutric\schwarzric
\gdef\schwarzoutlevi{\outlevi[\schwarzs]}
\gdef\euclideane{e}
\gdef\eukoutg{\outg[\euclideane]}
\NewTensor\sphg{\Omega}
\gdef\trans#1{{#1}^{\text t}}
\gdef\transg{\trans g}
\gdef\transh{\trans{h\hspace{-.05em}}\hspace{.05em}}
\gdef\eflapsymbol{f}
\NewTensor*\eflap[\!]\eflapsymbol
\NewTensor*\ewlap[\!]\lambda
\NewTensor*\efjac[\!]{\mathcal f}
\NewTensor*\ewjac[\!]<\!>\gamma
\gdef\deform#1{{#1}^{\text d}}
\NewDocumentCommand\rad{od<>os}
{\IfValueTF{#1}{\rad<#2>[#1]}
 {\IfDisplaystyleTF{\left|\outx<#2>[#3]\right|}{|\outx<#2>[#3]|}}}
\gdef\time{{\boldsymbol t}}
\gdef\volume#1{\IfDisplaystyleTF{\left|#1\right|}{|#1|}}
\undef\c\NewDocumentCommand\c{sG{c}}{\IfBooleanTF{#1}{#2}{\IndexSymbol*#2}}
\NewDocumentCommand\Cof{G{C}d<>d()}
{\IfValueTF{#3}{\Cof{#1}<\IfValueTF{#2}{#2,}\relax{#3}>}{\CofSnd{#1\IfValueTF{#2}{[#2]}\relax}}}
\NewDocumentCommand\CofSnd{G{C}d<>o}
{\IfValueTF{#3}{\CofSnd{#1}<\IfValueTF{#2}{#2,}\relax{#3}>}{\mathop{#1\IfValueTF{#2}{(#2)}\relax}}}%
\NewTensor\M[\hspace{-.05em}]<\hspace{-.05em}>\Sigma
\NewTensor[\outsymbol]\outM[\hspace{-.025em}]{\textrm M}
\NewTensor[\unisymbol]\uniM[\hspace{-.025em}]{\textrm M}
\NewTensor\rnu[\!]u
\NewTensor\tnu[\hspace{-.05em}] w
\NewTensor\lapse[\!]u
\NewTensor[\outsymbol]\ralpha[\hspace{-.05em}]\alpha
\NewTensor\rbeta[\!]\beta
\NewTensor[\outsymbol]\outbeta[\!]\beta
\NewTensor[\MakeSymbol<>\vec]\centerz[\!]<\hspace{-.05em}>z
\NewTensor[\MakeSymbol<\outsymbol>\vec]\outcenter[\!]<\hspace{-.05em}>Z
\undef\metric\NewTensor[\newmathcal]\metric[\!\!]<\hspace{.05em}>g
\let\g\metric
\NewTensor[\outsymbol[\newmathcal]]\outmetric<\hspace{.05em}>g
\let\outg\outmetric
\NewTensor[\unisymbol[\newmathcal]]\unimetric[\!\!]<\hspace{.05em}>g
\let\unig\unimetric
\NewTensor[\newmathcal]\riem R
\NewTensor[\outsymbol[\newmathcal]]\outriem R
\NewTensor[\normalfont\textrm]\ricci{Ric}
\let\ric\ricci
\let\Ric\ricci
\NewTensor[\outsymbol[\normalfont\textrm]]\outricci{Ric}
\let\outric\outricci
\NewTensor[\unisymbol[\normalfont\textrm]]\uniricci{Ric}
\let\uniric\uniricci
\NewTensor[\newmathcal]\scalar S
\NewTensor[\outsymbol[\newmathcal]]\outscalar S
\let\outsc\outscalar
\NewTensor[\unisymbol[\newmathcal]]\uniscalar S
\let\unisc\uniscalar
\NewTensor[\newmathcal]\sectional K
\NewTensor[\newmathcal]\gauss G
\NewTensor[\textrm]\II k
\let\zFund\k
\NewTensor[\tracefree]\IItrf[\!]{\textrm k}
\let\zFundtrf\ktrf
\NewTensor[\outsymbol[\textrm]]\outII k
\let\outzFund\outk
\NewDocumentCommand\troutzFund{D<>{}O{}}{\tr<#1>[#2]\hspace{.05em}\outzFund[#2]}
\NewDocumentCommand\outzFundnu{D<>{}O{}}{\mathop{{\outzFund[#2]_{\nu<#1>[#2]}}}}
\NewDocumentCommand\outzFundnunu{D<>{}O{}}{\mathop{{\outzFund[#2]_{\nu<#1>[#2]\nu<#1>[#2]}}}}
\NewTensor[\mathcal]\mc[\!]<\hspace{-.05em}> H
\let\H\mc
\NewTensor[\outsymbol[\mathcal]]\outmc[\!]<\hspace{-.05em}>H
\let\outH\outmc
\let\oldnu\nu
\NewTensor\normal[\!]\oldnu
\let\nu\normal
\NewTensor[\outsymbol]\timenormal[\!]\vartheta
\let\tv\timenormal
\let\outnu\timenormal
\NewTensor*\laplace[\!]\Delta
\NewTensor[\outsymbol]*\outlaplace[\!]\Delta
\NewTensor\Hess[\!]{\text{Hess}}
\NewTensor\Hesstrf[\!]{\text H\tracefree{\text{es}}\text s}
\NewTensor[\outsymbol]\outHess[\!]{\text{Hess}}
\undef\div\NewTensor\div[\!]{\text{div}}
\NewTensor[\outsymbol]\outdiv[\!]{\text{div}}
\NewTensor[\unisymbol]\unidiv[\!]{\text{div}}
\NewTensor*\tr[\!]{\text{tr}}
\NewTensor[\outsymbol]*\outtr[\!]{\text{tr}}
\NewTensor[\unisymbol]*\unitr[\!]{\text{tr}}
\NewTensor\levi<\MakeSymbol<>{\!}>{\MakeSymbol<\Gamma>\nabla}
\NewTensor[\outsymbol]\outlevi<\MakeSymbol<>{\!}>{\MakeSymbol<\Gamma>\nabla}
\NewTensor[\unisymbol]\unilevi<\MakeSymbol<>{\!}>{\MakeSymbol<\Gamma>\nabla}
\NewTensor[\outsymbol]*\vol{\text{vol}}
\RenewTensor*\exp{\text{exp}}
\NewTensor[\outsymbol]*\outexp{\text{exp}}
\NewTensor[\unisymbol]*\uniexp{\text{exp}}
\let\Vol\vol
\NewTensor*\mug[\!]\mu
\gdef\eukmug{\mug[\euclideane]}
\NewTensor[\outsymbol]*\outmug[\!]\mu
\NewTensor[\outsymbol]\momentum[\!]\Pi
\NewTensor[\outsymbol]\einstein{{\normalfont\textrm G}}
\NewTensor[\outsymbol]\outmomden[\!]J
\NewTensor[\outsymbol]\outenden[\hspace{-.05em}]\varrho
\NewTensor[\textrm]*\jacobiext[\hspace{-.05em}]<\hspace{-.05em}> L
\NewDocumentCommand\jacobit{G{\pm}d<>od<>s}
{\IfValueTF{#2}{\jacobit{#1}[#3]<#2>}
 {\IndexSymbol*<\hspace{-.05em}>{\textrm L}[#3]<#4>_{#1}^{\!t}}}
\NewDocumentCommand\jacobipm{G{\pm}d<>od<>s}
{\IfValueTF{#2}{\jacobipm{#1}[#3]<#2>}
 {\IndexSymbol*<\hspace{-.05em}>{\textrm L}[#3]<#4>_{#1}}}
\NewDocumentCommand\ajacobipm{G{\gewicht}}{\IndexSymbol*[\hspace{-.05em}]<\hspace{-.05em}>{\textrm L_{#1}^{\!\!a}}}
\NewDocumentCommand\trzd{smm}{{#2}\odot{#3}}
\NewDocumentCommand\trtr{smm}
{\ifthenelse{\equal{#2}{#3}}{\left|#2\right|_{\g*}^2}{\tr(\trzd{#2}{#3})}}%
\NewDocumentCommand\outc{G{c}}{\IndexSymbol(\outsymbol)*{#1}}\let\oc\outc
\gdef\outve{\IndexSymbol\ve}
\NewDocumentCommand\outck{d<>o}{\outc<#1>[#2]_{\,\outzFund*}\vphantom{\outc_{\,\outzFund*}}}
\gdef\outtr{\IndexSymbol(\outsymbol)[\!]{\text{tr}}}
\NewDocumentCommand\outtrzd{D<>{}O{}mm}
{\mathop{{#3\IndexSymbol(\outsymbol)[#2]<#1>\odot{#4}}}}
\NewDocumentCommand\outtrtr{D<>{}O{}mm}
{\mathop{{\ifthenelse{\equal{#3}{#4}}
 {{\left|#3\right|_{\outg<#1>[#2]}^2}}
 {\mathop{{\outtr<#1>[#2](\outtrzd<#1>[#2]{#3}{#4})}}}}}}
\NewTensor\mHaw{{\normalfont m_{\text{H}}}}
\NewTensor[\MakeSymbol<\outsymbol>\vec]\impuls P
\undef\d
\NewDocumentCommand\d{s}{\IfBooleanTF{#1}\relax{\mathop{}\!}\mathrm d}%
\newcommand\pullback[1]{{#1}^*}
\DeclareMathOperator\id{id}
\DeclareMathOperator\lin{lin}
\DeclareMathOperator\graph{graph}
\gdef\graphnu#1#2{\graph^{\hspace{.05em}#1\hspace{-.15em}}#2}
\newcommand\R{\normalfont\mathds{R}}																
\newcommand\N{\normalfont\mathds{N}}																
\newcommand\X{\normalfont\mathfrak{X}}															
\newcommand\Lp{\normalfont\textrm L}
\newcommand\Wkp{\normalfont\textrm W}
\newcommand\Hk{\normalfont\textrm H}
\newcommand\Ck{\normalfont\mathcal C}
\NewDocumentCommand\sphere{t^}{\IfBooleanTF{#1}{\mathds S^}{\mathds S^2}}
\newcommand\ve{\varepsilon}																					
\usepackage{urwchancal}																							
\DeclareMathAlphabet{\mathcal}{OT1}{pzc}{m}{n}
\let\newmathcal\mathcal
\NewDocumentCommand\normalbrace{t+ssmss}
{\IfBooleanTF{#1}\relax\hbox{
 {\normalfont(}\IfBooleanTF{#2}\relax{\hspace{-.05em}}
 \IfBooleanTF{#3}{\hspace{-.1em}}\relax#4
 \IfBooleanTF{#5}\relax{\hspace{.05em}}
 \IfBooleanTF{#6}{\hspace{.1em}}\relax{\normalfont)}}}